\documentclass[12pt]{article}
\usepackage{blindtext}
\usepackage[a4paper
, total={7in, 9in}
]{geometry}

\makeatletter

\@addtoreset{equation}{section}
\makeatother

\usepackage[cp1251]{inputenc}
\usepackage[OT1]{fontenc}
\usepackage[english]{babel}
\usepackage[tbtags]{amsmath}
\usepackage{amsfonts, amsmath, amsthm, amssymb, graphicx, comment}
\usepackage{cite}

\graphicspath{{figures/}}

\newcommand{\figout}[1]
{#1}

\newtheorem{lemma}{Lemma}
\newtheorem{theorem}{Theorem}

\newtheorem{corollary}{Corollary}
\newtheorem{proposition}{Proposition}
\newtheorem{remark}{Remark}

\theoremstyle{definition}

\def\a{\alpha}
\def\b{\beta}
\def\d{\delta}
\def\f{\varphi}
\def\p{\psi}
\def\s{\sigma}

\newcommand{\D}{\Delta}
\def\lam{\lambda}
\def\g{\gamma}
\def\eps{\varepsilon}
\def\R{{\mathbb R}}

\def\then{\quad\Rightarrow\quad}
\def\vh{\vec h}
\def\vH{\vec H}

\def\aen{\mbox{\ae}}

\def\A{\mathcal{A}}
\def\B{\mathcal{B}}

\def\tN{\widetilde{N}}
\def\tlam{\widetilde{\lam}}

\def\tq{\widetilde{q}}

\def\ts{\widetilde{\sigma}}

\def\hq{\widehat{q}}
\def\dh{\dot{h}}
\def\dq{\dot{q}}
\def\bt{\bar{t}}
\def\bq{\bar{q}}

\def\bs{\bar{\sigma}}

\def\sL{sub-Lo\-rent\-zi\-an } 
\def\SL{Sub-Lo\-rent\-zi\-an }

\newcommand{\spann}{\operatorname{span}\nolimits}
\newcommand{\Exp}{\operatorname{Exp}\nolimits}
\newcommand{\Lip}{\operatorname{Lip}\nolimits}
\newcommand{\sgn}{\operatorname{sgn}\nolimits}

\newcommand{\VEC}{\operatorname{Vec}\nolimits}

\newcommand{\intt}{\operatorname{int}\nolimits}
\newcommand{\cl}{\operatorname{cl}\nolimits}
\newcommand{\sn}{\operatorname{sn}\nolimits}
\newcommand{\cn}{\operatorname{cn}\nolimits}
\newcommand{\dn}{\operatorname{dn}\nolimits}
\newcommand{\am}{\operatorname{am}\nolimits}
\newcommand{\E}{\operatorname{E}\nolimits}
\newcommand{\cut}{\operatorname{cut}\nolimits}

\def\tt{\mathbf{t}}
\def\tcut{t_{\cut}}
\def\Mc{\overset{\circ}{M}}

\newcommand{\eq}[1]{$(\protect\ref{#1})$}
\newcommand{\be}[1]{\begin{equation}\label{#1}}
\newcommand{\ee}{\end{equation}}
\newcommand{\pder}[2]{\frac{\partial \, #1}{\partial \, #2} }
\newcommand{\map}[3]{#1 \, : \, #2 \to #3}
\newcommand{\restr}[2]{\left. #1 \right|_{#2}}

\newcommand{\onefiglabelsizen}[4]
{
\begin{figure}[htbp]
\begin{center}
\includegraphics[height=#4cm]{#1}
\\
\parbox[t]{0.7\textwidth}{\caption{#2}\label{#3}}
\end{center}
\end{figure}
}

\newcommand{\twofiglabelsizeh}[8]
{
\begin{figure}[htbp]
\includegraphics[height=#4cm]{#1}
\hfill
\includegraphics[height=#8cm]{#5}
\\
\parbox[t]{0.4\textwidth}{\caption{#2}\label{#3}}
\hfill
\parbox[t]{0.4\textwidth}{\caption{#6}\label{#7}}
\end{figure}
}

\author{Yu. L. Sachkov\\
Ailamazyan Program Systems Institute of RAS\\
RUDN University\\
e-mail: yusachkov@gmail.com}

\title{Flat sub-Lorentzian structures on    Martinet distribution
\footnote{
Work in Section 3 supported by Russian Scientific Foundation, grant 22-11-00140, https://rscf.ru/project/22-11-00140/.
Work in Section 4 supported by   the Theoretical Physics and Mathematics Advancement Foundation 
"BASIS", 
grant 23-7-1-16-1.
}
}

\begin{document}

\maketitle

\begin{abstract}
Two flat sub-Lorentzian problems on the  Martinet  distribution are studied. For the first one, the attainable set has a nontrivial intersection with the Martinet plane, but for the second one it does not. Attainable sets, optimal trajectories, sub-Lorentzian distances and spheres are described.
\end{abstract}


%\newpage

\tableofcontents

%\newpage

\section{Introduction}
Sub-Riemannian geometry studies manifolds $M$ in which the distance between points $q_0, q_1 \in M$ is the infimum of the lengths of all curves tangent to a given distribution and connecting $q_0$ to $q_1$ \cite{mont_book, ABB}. In more detail, for the distribution $\D \subset TM$ in each space $\D_q \subset T_qM$ the scalar product $g_q$ is given, and the length of the curve $q(t)$, $t \in [0, t_1]$, tangent $\D$, is measured as in Riemannian geometry: $l(q(\cdot)) = \int_0^{t_1} \sqrt{g(\dot q(t))} dt$.

If in each space $\D_q \subset T_qM$ we define a non-degenerate quadratic form $g_q$ of index 1, then a sub-Lorentzian structure $(\D, g)$ will be defined on the manifold $M$.
Here the natural problem is to find the longest relative to $g$ curve connecting given points.
Sub-Lorentzian geometry strives to build a theory similar to the rich theory of sub-Riemannian geometry, and is at the beginning of its development. The foundations of sub-Lorentzian geometry were laid in the works of M.~Grochowski~\cite{groch2,groch3,groch9,groch11,groch4,groch6}, see also \cite{chang_mar_vas,grong_vas,groch_med_war, kor_mar}.

Just as in sub-Riemannian geometry, the simplest sub-Lorentzian problem arises on the Heisenberg group; it has been fully studied \cite{groch4,sl_heis}.
The next most important model of sub-Riemannian geometry after the Heisenberg group arises on the  Martinet distribution \cite{mont_book, ABCK, ABB, UMN2}.

The purpose of this work is to consider two flat sub-Lorentzian problems on the  Martinet distribution: to describe the optimal synthesis, distance and spheres. In the first problem, the future cone has a non-trivial intersection with the tangent space to the Martinet surface; in the second case this intersection is trivial. Accordingly, in the first case the sub-Lorentzian geometry is much more complicated, see Conclusion.

The structure of this work is as follows. In Section \ref{sec:SL} we recall the basic facts of \sL geometry required in the sequel. 

The main Sections \ref{sec:P1} and  \ref{sec:P2} are devoted respectively to the first and the second flat \sL problems on the Martinet  distribution; they have identical structure as follows. First we find an invariant set (a candidate attainable set) via the geometric statement of Pontryagin maximum principle. Then we describe explicitly abnormal and normal extremal trajectories; normal trajectories are parametrized by the \sL exponential mapping. We prove diffeomorphic properties of the exponential mapping via Hadamard's global diffeomorphism theorem. On this basis we show that the above-mentioned  invariant set is indeed the attainable set, and prove a theorem on existence of optimal trajectories. After that we study optimality of extremal trajectories, which yields an optimal synthesis. We complete our study by describing main properties of \sL distance and sphere.

In the concluding Sec. \ref{sec:conclude} we discuss the results obtained for two problems.

\section{\SL geometry}\label{sec:SL}
A \sL structure on a smooth manifold $M$ is a pair $(\D, g)$ consisting of a vector distribution $\D \subset TM$ and a Lorentzian metric $g$ on $\D$, i.e., a nondegenerate quadratic form $g$ of index~1.  
Let us recall some basic definitions of \sL geometry.  
A vector $v \in T_qM$, $q \in M$, is called horizontal if $v \in \D_q$. A horizontal vector $v$ is called:
\begin{itemize}
\item
timelike if $g(v)<0$,
\item
spacelike if $g(v)>0$ or $v = 0$,
\item
lightlike if $g(v)=0$ and $v \neq 0$,
\item
nonspacelike if $g(v)\leq 0$.
\end{itemize}
A Lipschitzian curve in $M$ is called timelike if it has timelike velocity vector a.e.; spacelike, lightlike and nonspacelike curves are defined similarly.

A time orientation $X$ is an arbitrary timelike vector field in $M$. A nonspacelike vector $v \in \D_q$ is future directed if $g(v, X(q))<0$, and past directed if $g(v, X(q))>0$. 

A future directed timelike curve $q(t)$, $t \in [0, t_1]$, is called arclength paramet\-ri\-zed if $g(\dot q(t), \dot q(t)) \equiv - 1$. Any future directed timelike curve can be parametrized by arclength, similarly to the arclength parametrization of a horizontal curve in sub-Riemannian geometry.

The length of a nonspacelike curve $\g \in \Lip([0, t_1], M)$ is 
$$
l(\g) = \int_0^{t_1} |g(\dot \g, \dot \g)|^{1/2} dt.
$$

For points $q_1, q_2 \in M$ denote by $\Omega_{q_1q_2}$ the set of all future directed nonspacelike curves in $M$ that connect $q_1$ to $q_2$. In the case $\Omega_{q_1q_2} \neq \emptyset$ denote the \sL distance from the point $q_1$ to the point $q_2$ as
\be{d}
d(q_1, q_2) = \sup \{l(\g) \mid \g \in \Omega_{q_1q_2}\}.
\ee
And if $\Omega_{q_1q_2} = \emptyset$ then $d(q_1, q_2) := 0$.
A future directed nonspacelike curve $\g$ is called a \sL length maximizer if it realizes the supremum in \eq{d} between its endpoints $\g(0) = q_1$, $\g(t_1) = q_2$.

The causal future of a point $q_0 \in M$ is the set $J^+(q_0)$ of points $q_1 \in M$ for which there exists a future directed nonspacelike curve $\g$ that connects $q_0$ and $q_1$. 

Let $q_0 \in M$, $q_1 \in J^+(q_0)$. The search for \sL length maximizers that connect $q_0$ with~$q_1$ reduces to the search for future directed nonspacelike curves $\g$ that solve the problem
\be{lmax}
l(\g) \to \max, \qquad \g(0) = q_0, \quad \g(t_1) = q_1.
\ee

A set of vector fields $X_1, \dots, X_k \in \VEC(M)$ is an orthonormal frame for a \sL structure $(\D, g)$ if for all $q \in M$
\begin{align*}
&\D_q = \spann(X_1(q), \dots, X_k(q)),\\
&g_q(X_1, X_1) = -1, \qquad g_q(X_i, X_i) = 1, \quad i = 2, \dots, k, \\
&g_q(X_i, X_j) = 0, \quad i \neq j. 
\end{align*}
Assume that time orientation is defined by a timelike vector field $X \in \VEC(M)$ for which $g(X, X_1) < 0$ (e.g., $X = X_1$). Then the   \sL problem for the \sL structure with the orthonormal frame $X_1, \dots, X_k$ is stated as the following optimal control problem:
\begin{align*}
&\dot q = \sum_{i=1}^k u_i X_i(q), \qquad q \in M, \\
&u \in U = \left\{(u_1, \dots, u_k) \in \R^k \mid u_1 \geq \sqrt{ u_2^2 + \dots + u_k^2}\right\},\\
&q(0) = q_0, \qquad q(t_1) = q_1, \\
&l(q(\cdot)) = \int_0^{t_1} \sqrt{u_1^2 - u_2^2 -  \dots - u_k^2} \, dt  \to \max.
\end{align*}

\begin{remark}
The \sL length is preserved under monotone Lipschitzian time reparametrizations $t(s)$, $s \in [0, s_1]$. Thus if $q(t)$, $t \in [0, t_1]$, is a \sL length maximizer, then so is any its reparametrization $q(t(s))$, $s \in [0, s_1]$. 

In this paper we choose primarily the following parametrization of trajectories: the arclength parametrization ($u_1^2 - u_2^2 - \cdots - u_k^2 \equiv 1$) for timelike trajectories, and the parametrization with $u_1(t) \equiv 1$ for future directed lightlike trajectories. 
\end{remark}

\section{The first problem}\label{sec:P1} 
Let $M = \R^3_{x, y, z}$, $X_1 = \pder{}{x}$, $X_2 = \pder{}{y} + \frac{x^2}{2} \pder{}{z}$. The distribution $\D = \spann(X_1, X_2)$ is called the  Martinet  distribution \cite{mont_book, ABCK, ABB, UMN2}. 
The plane $\Pi = \{x=0\}$ is called the Martinet surface.
The distribution~$\D$ has growth vector $(2,3)$ outside of $\Pi$, and growth vector $(2, 2, 3)$ on $\Pi$. This distribution is called flat since the Lie algebra generated by the vector fields $X_1$, $X_2$ is a Carnot algebra (Engel algebra), the nonzero Lie brackets of these vector fields are: $[X_1, X_2] = x X_3$, $[X_1, x X_3] = X_3 := \pder{}{z}$.

In this section we study a \sL problem in which the interior of the future cone intersects nontrivially with the tangent space to the Martinet plane $\Pi$. 

\subsection{Problem statement}

The first flat sub-Lorentzian problem on the Martinet  distribution is stated as the following optimal control problem \cite{notes, intro}:
\begin{align}
&\dot q = u_1 X_1 + u_2 X_2, \qquad q \in M, \label{pr11} \\
&u= (u_1, u_2) \in U_1 = \{u_2 \geq |u_1|\}, \label{pr12}\\
&q(0) = q_0 = (0, 0, 0), \qquad q(t_1) = q_1, \label{pr13}\\
&l = \int_0^{t_1} \sqrt{u_2^2-u_1^2}dt \to \max, \label{pr14}
\end{align}
see Fig. \ref{fig:U1}.

\figout{
\onefiglabelsizen
{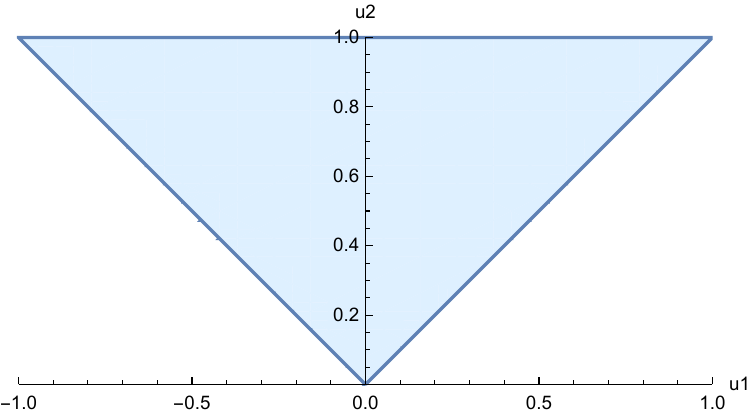}{The set $U_1$}{fig:U1}{4}
}

\subsection{Invariant set}\label{subsec:invar_set}
In this subsection we compute an invariant set $\B_1$ of system \eq{pr11}, \eq{pr12}. Later, in Th. \ref{th:att_set}, we prove that $\B_1$ is the attainable set $\A_1$ of system \eq{pr11}, \eq{pr12} from the point $q_0$ for arbitrary nonnegative time (the causal future of the point $q_0$). 

By the geometric statement of Pontryagin maximum principle (PMP) for free time (\cite{notes}, Th.~12.8), if a trajectory $q(t)$ corresponding to a control $u(t)$, $t \in [0, t_1]$, satisfies the inclusion $q(t_1) \in \partial \A_1$, then there exists a Lipschitzian curve $\lam_t \in T^*_{q(t)}M$, $\lam_t \neq 0$, $t \in [0, t_1]$, such that
\begin{align}
&\dot \lam_t = \vh_{u(t)}(\lam_t), \label{Ham01}\\
&h_{u(t)}(\lam_t) = \max_{u \in U_1} h_u(\lam_t), \label{max1}\\
&h_{u(t)}(\lam_t) = 0 \nonumber
\end{align}
for almost all $t \in [0, t_1]$. Here $h_u(\lam) = u_1 h_1(\lam) + u_2 h_2(\lam)$, $h_i(\lam) = \langle \lam, X_i(\pi(\lam))\rangle$, $i = 1, 2$, and $\map{\pi}{T^*M}{M}$ is the canonical projection of the cotangent bundle, $\pi(\lam) = q$, $\lam \in T^*_qM$. Moreover, $\vh_{u}(\lam)$ is the Hamiltonian vector field on the cotangent bundle $T^*M$ with the Hamoltonian $h_u(\lam)$.

We have $[X_1, X_1] = x X_3$, $X_3 = \pder{}{z}$, and if we denote $h_3(\lam) = \langle \lam, X_3(\pi(\lam))\rangle$, then the Hamiltonian system \eq{Ham01} reads
$$
\dh_1 = - u_2 xh_3, \quad
\dh_2 = u_1xh_3, \quad
\dh_3 = 0, \quad
\dq = u_1X_1+u_2X_2.
$$
The maximality condition \eq{max1} implies that up to reparameterization there can be two cases:
\begin{itemize}
\item[a)]
$u (t) = (\pm 1, 1)$, 
\item[b)]
$u (t) = (0, 1)$, $x(t) = 0$.
\end{itemize}
Take any $0 \leq t_1 \leq t_2$ and compute trajectories with one switching corresponding to the following controls:

1) Let 
$$u(t) = \begin{cases} 
(1, 1), & t \in [0, t_1], \\
(-1, 1), & t \in [t_1, t_2].
\end{cases}
$$
Then $x(t) = t$, $y(t) = t$, $z(t) = t^3/6$ for $t \in [0, t_1]$, 
$x(t) = 2t_1-t$, $y(t) = t$, $z(t) = t_1^3/6 + (4t_1^2(t-t_1)-2t_1(t^2-t_1^2)+(t^3-t_1^3)/3)$ for $t \in [t_1, t_2]$, thus $x(t_2) = 2t_1-t_2$, $y(t_2) = t_2$, $z(t_2) = - t_1^3 + 2t_1^2t_2-t_1t_2^2+t_2^3/6$. Thus the endpoint $q(t_2)$ satisfies the equality
\be{zS1}
z = (-3x^3+3x^2y+3xy^2+y^3)/24.
\ee

2) Let 
$$u(t) = \begin{cases} 
(-1, 1), & t \in [0, t_1], \\
(1, 1), & t \in [t_1, t_2].
\end{cases}
$$
Then $x(t) = -t$, $y(t) = t$, $z(t) = t^3/6$ for $t \in [0, t_1]$, 
$x(t) = t -2t_1$, $y(t) = t$, $z(t) = ((t^3-t_1^3)/3 - 2t_1(t^2-t_1^2)+4t_1^2(t-t_1))/2$ for $t \in [t_1, t_2]$, thus $x(t_2) = t_2-2t_1$, $y(t_2) = t_2$, $z(t_2) = - t_1^3 + 2t_1^2t_2-t_1t_2^2+t_2^3/6$. Thus the endpoint $q(t_2)$ satisfies the equality
\be{zS2}
z = (3x^3+3x^2y-3xy^2+y^3)/24.
\ee

3) Let 
$$u(t) = \begin{cases} 
(0, 1), & t \in [0, t_1], \\
(1, 1), & t \in [t_1, t_2].
\end{cases}
$$
Then $x(t) = 0$, $y(t) = t$, $z(t) = 0$ for $t \in [0, t_1]$, 
$x(t) = t -t_1$, $y(t) = t$, $z(t) = (t-t_1)^3/6$ for $t \in [t_1, t_2]$, thus $x(t_2) = t_2-t_1$, $y(t_2) = t_2$, $z(t_2) = (t_2-t_1)^3/6$. Thus the endpoint $q(t_2)$ satisfies the equality
\be{zS3}
z = x^3/6.
\ee

4) Finally, let 
$$u(t) = \begin{cases} 
(0, 1), & t \in [0, t_1], \\
(-1, 1), & t \in [t_1, t_2].
\end{cases}
$$
Then $x(t) = 0$, $y(t) = t$, $z(t) = 0$ for $t \in [0, t_1]$, 
$x(t) = t_1 -t$, $y(t) = t$, $z(t) = (t-t_1)^3/6$ for $t \in [t_1, t_2]$, thus $x(t_2) = t_1-t_2$, $y(t_2) = t_2$, $z(t_2) = (t_2-t_1)^3/6$. Thus the endpoint $q(t_2)$ satisfies the equality
\be{zS4}
z = -x^3/6.
\ee
Consider the surfaces $S_1$--$S_4$ given by Eqs. \eq{zS1}--\eq{zS4} respectively,
\begin{align*}
S_1 \ &: \ z = (-3x^3+3x^2y+3xy^2+y^3)/24, \qquad &x \geq 0,\\
S_2 \ &: \ z = (3x^3+3x^2y-3xy^2+y^3)/24, \qquad &x \leq 0,\\
S_3 \ &: \ z =  x^3 /6, \qquad &x \geq 0,\\
S_4 \ &: \ z =  -x^3 /6, \qquad &x \leq 0.
\end{align*}
Introduce the homogeneous coordinates on the set $\{y \neq 0\}$ induced by the one-parameter group of dilations \eq{dilat}:
\be{xieta}
\xi = \frac xy, \qquad \eta = \frac{24z-3x^2y-y^3}{24y^3}.
\ee
Then the surfaces $S_1$--$S_4$ are given as follows:
\begin{align*}
S_1 \ &: \ \eta = \frac{\xi(1-\xi^2)}{8} =: \f_1(\xi), \qquad &\xi \in [0, 1],\\
S_2 \ &: \ \eta = \frac{\xi(\xi^2-1)}{8} =: \f_2(\xi) = \f_1(-\xi), \qquad &\xi \in [-1, 0],\\
S_3 \ &: \ \eta = \frac{\xi^3}{6} - \frac{3\xi^2 +1}{24} =: \f_3(\xi), \qquad &\xi \in [0, 1],\\
S_4 \ &: \ \eta = -\frac{\xi^3}{6} - \frac{3\xi^2 +1}{24} =: \f_4(\xi)  = \f_3(-\xi), \qquad &\xi \in [-1, 0].
\end{align*}
The surface $\cup_{i=1}^4 S_i$ bounds a domain
$$
\B_1 = \begin{cases}
\f_3(\xi) \leq \eta \leq \f_1(\xi), \qquad& 0 \leq \xi\leq 1, \\
\f_4(\xi) \leq \eta \leq \f_2(\xi), \qquad& -1 \leq \xi\leq 0,
\end{cases}
$$
see Figs. \ref{fig:A3D1}--\ref{fig:A2D}.

\figout{
\twofiglabelsizeh
{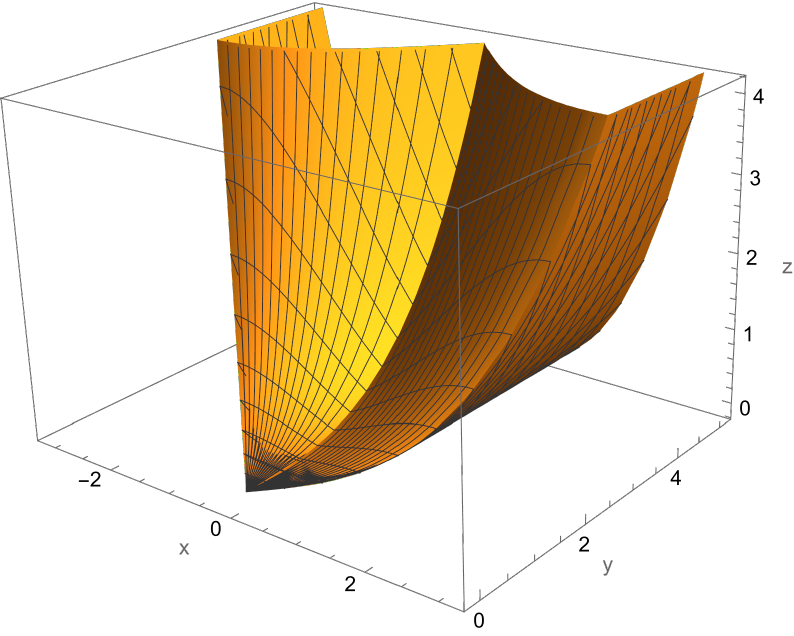}{The boundary of $\B_1$}{fig:A3D1}{5}
{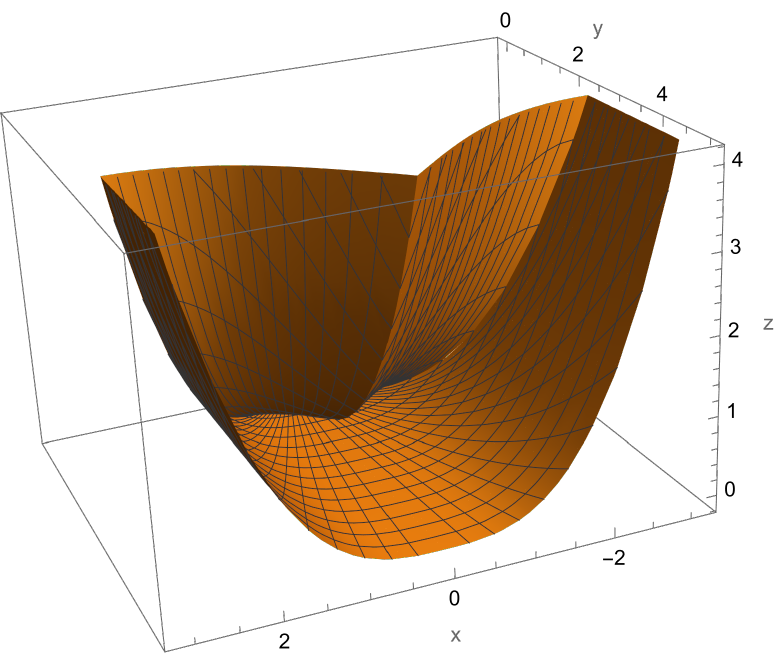}{The boundary of $\B_1$}{fig:A3D2}{5}

\twofiglabelsizeh
{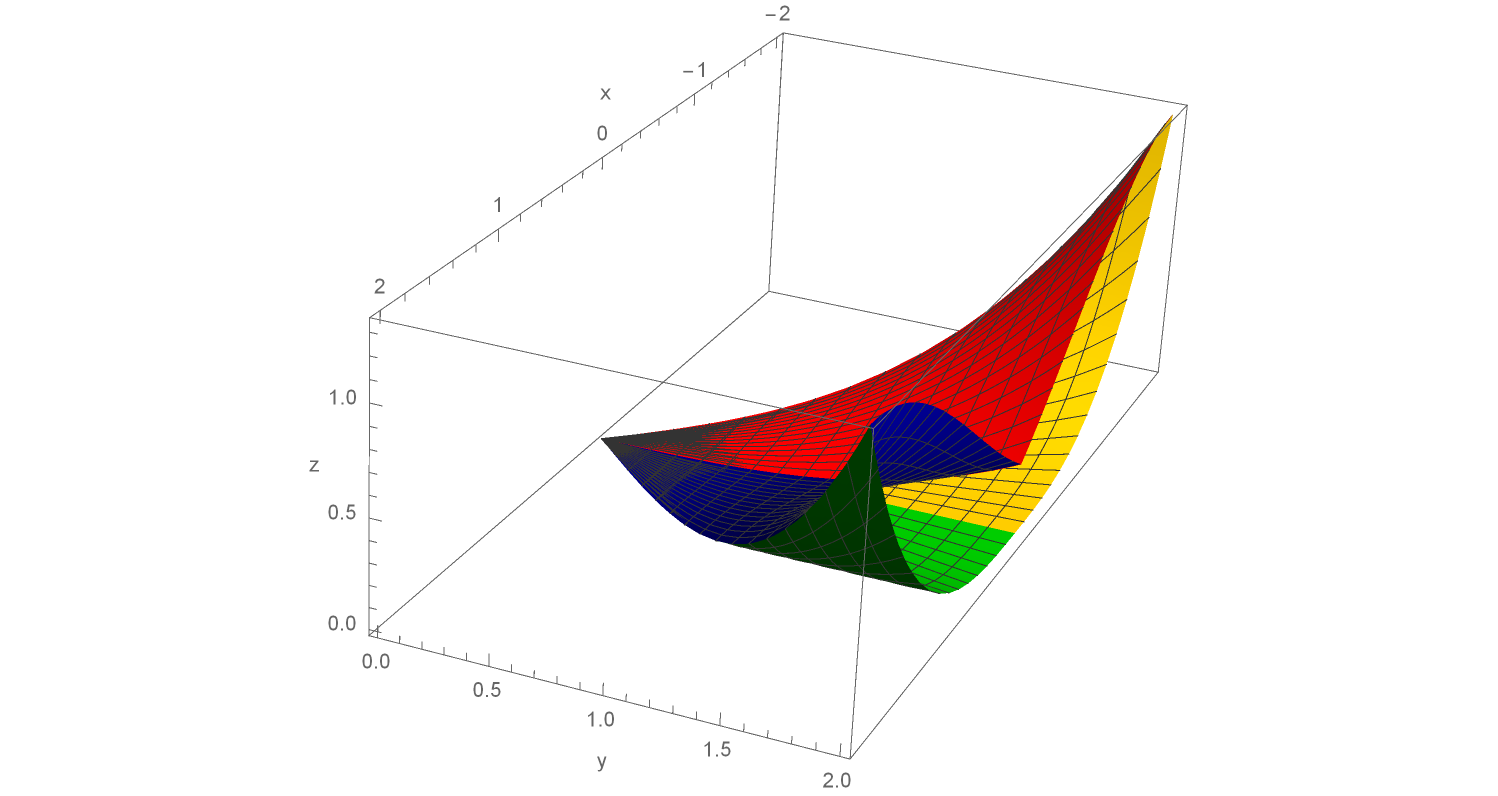}{The boundary of $\B_1$}{fig:A3D3}{6}
{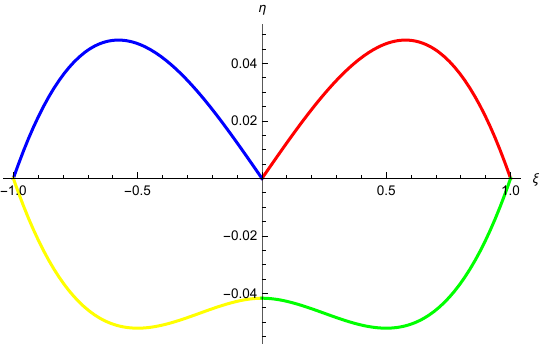}{The boundary of $\B_1$ projected to $(\xi, \eta)$}{fig:A2D}{5}
}

Recall that $\A_1$ is the attainable set of system \eq{pr11}, \eq{pr12} from the point $q_0$ for arbitrary nonnegative time (the causal future of the point $q_0$). 

\begin{proposition}\label{propos:AinB1}
The set $\B_1$ is an invariant domain of system \eq{pr11}, \eq{pr12}. Moreover, $\A_1 \subset B_1$.
\end{proposition} 
\begin{proof}
Direct computation shows that on each of the surfaces $S_1$--$S_4$
the vector field $u_1X_1+u_2X_2$, $(u_1, u_2) \in U_1$, is directed inside the domain $\B_1$. Since $q_0 \in \B_1$, then $\A_1 \subset \B_1$.
\end{proof}

We show in Th. \ref{th:att_set} that $\A_1 = \B_1$.

\subsection{Extremal trajectories}
Introduce the family of Hamiltonians of Pontryagin maximum principle (PMP) \cite{PBGM, notes} $h_u^{\nu}(\lam) = u_1 h_1(\lam) + u_2 h_2(\lam) - \nu \sqrt{u_2^2-u_1^2}$, where $\lam \in T^*M$, $(u_1, u_2) \in U_1$, $\nu \in \{-1, 0\}$. By PMP (Th.~12.10 \cite{notes}), if $q(t)$, $t \in [0, t_1]$, is an optimal trajectory in problem \eq{pr11}--\eq{pr14}, then there exist a Lipschitzian curve $\lam_t \in T_{q(t)}^* M$, $t \in [0, t_1]$, and a number  $\nu \in \{-1, 0\}$ such that
\begin{align*}
&\dot \lam_t = \vh^{\nu}_{u(t)}(\lam_t), \\
&h^{\nu}_{u(t)}(\lam_t) = \max_{v \in U_1} h_v(\lam_t), \\
&(\lam_t, \nu) \neq (0, 0)
\end{align*}
for almost all $t \in [0, t_1]$.

\subsubsection{Abnormal extremal trajectories}

If $\nu = 0$, then the control satisfies, up to reparameterization, the conditions
\begin{itemize}
\item[a)]
$u (t) = (\pm 1, 1)$, 
\item[b)]
$u (t) = (0, 1)$, $x(t) = 0$,
\end{itemize}
and has up to one switching. These trajectories were computed in Subsec. \ref{subsec:invar_set}, they form the boundary of the candidate attainable set $\B_1$.

\begin{remark}
Abnormal trajectories starting from an arc on the plane $\Pi$  change their causal type: first they are timelike (when belong to $\Pi$), then lightlike.
The remaining extremal trajectories preserve the causal type.
\end{remark}

\subsubsection{Normal extremals}\label{subsec:norm_extr}
If $\nu = -1$, then extremals satisfy the Hamiltonian system with the Hamiltonian $H = (h_1^2-h_2^2)/2$, $h_2 < - |h_1|$:
\be{Ham0}
\dh_1 = h_2h_3x, \quad \dh_2 = h_1h_3x, \quad \dh_3 = 0, \quad \dot x = h_1, \quad \dot y = - h_2, \quad \dot z = - h_2 x^2/2.
\ee
We can choose arclength parameterization on normal extremal trajectories and thus assume that $H \equiv -1/2$. In the coordinates $h_1 = \sinh \p$, $h_2 = - \cosh \p$, $h_3 = c$; $\p, c \in \R$, the Hamiltonian system~\eq{Ham0} reads
\begin{align}
&\dot \p = - cx, \qquad \dot c = 0, \label{Ham1}\\
&\dot x = \sinh \p, \qquad \dot y = \cosh \p, \qquad \dot z = \cosh \p x^2/2. \label{Ham2}
\end{align}
This system has an energy integral $E = cx^2/2+ \cosh \p \in [1, + \infty)$. 

\begin{remark}
The normal Hamiltonian system \eq{Ham1}, \eq{Ham2} has a discrete symmetry --- reflection
\be{refl}
(\p, c, x, y, z) \mapsto (-\p, c, -x, y, z),
\ee
and a one-parameter family of symmetries --- dilations
\be{dilat}
(t, \p, c, x, y, z) \mapsto (\a t, \p, c/\a^2, \a x, \a y, \a^3 z), \qquad \a > 0. 
\ee

Moreover, the parallel translations
\be{paral}
(x, y, z) \mapsto (x, y+a, z + b), \qquad a, b \in \R,
\ee
are symmetries of the problem since their generating vector fields $\pder{}{y}$, $\pder{}{z}$ commute with the vector fields $X_1$, $X_2$ of the orthonormal frame. 
\end{remark}

1) If $c = 0$, then
\be{c=0}
\p \equiv \p_0, \qquad x = t \sinh \p_0, \qquad y = t \cosh \p_0, \qquad z = t^3/6 \cosh \p_0 \sinh^2 \p_0.
\ee 

If $c \neq 0$, then extremal trajectories in the  Martinet flat case are obtained by a linear change of variables from extremal trajectories of a left-invariant \sL problem on the Engel group \cite{SL_Engel2}.

2) Let $c = l^2 > 0$.

2.1) If $\p_0 = 0$, then $x=z \equiv 0$, $y = t$.

2.2) Let $\sgn \p_0 = \pm 1$, $E = \cosh \p_0 > 1$, $k = \sqrt{\frac{E-1}{E+1}} \in (0, 1)$, $k' = \sqrt{1-k^2}$, $m = l k'$, $\aen = 1/k'$, $\tau = \aen l t$. Then
\begin{align*}
&x = \pm \frac{2k}{m} \sn \tau, \\
&y = \frac 1m (2 \E(\tau) - k'^2 \tau), \\
&z = - \frac{2}{3m^3} \left(k'^2\tau + 2 k^2 \sn \tau \cn \tau \dn \tau - (1+k^2) \E(\tau)\right),
\end{align*}
where $\sn \tau$, $\cn \tau$, $\dn \tau$ are Jacobi's elliptic functions with modulus $k$, and $\E(\tau) = \int_0^{\tau} \dn^2 t dt $ is Jacobi's epsilon function \cite{whit_watson, lawden}. See Figs. \ref{fig:c1p0}--\ref{fig:c1p4}.

3) Let $c = - l^2 < 0$.

3.1) If $\p_0 = 0$, then $x=z \equiv 0$, $y = t$.

3.2) Let $\sgn \p_0 = \pm 1$, $E = \cosh \p_0 > 1$, $k = \sqrt{\frac{2}{1+E}} \in (0, 1)$, $k' = \sqrt{1-k^2}$, $m = k l$, $\tau = lt/k$. Then
\begin{align*}
&x = \pm \frac{2k'}{m} \frac{\sn \tau}{\cn \tau}, \\
&y = \frac 1m \left((2-k^2)\tau + 2 \frac{\dn \tau \sn \tau}{\cn \tau} - 2 \E(\tau)  \right), \\
&z = - \frac{2}{3m^3} \left(2k'^2\tau + (k^2-2) \E(\tau) + (k^2 + (k^2-2)\sn^2 \tau) \frac{\dn \tau \sn \tau}{\cn^3 \tau}\right),\\
&\tau \in [0, K(k)),
\end{align*}
where $K(k)$ is the complete elliptic integral of the second kind \cite{whit_watson, lawden}.

\figout{
\twofiglabelsizeh
{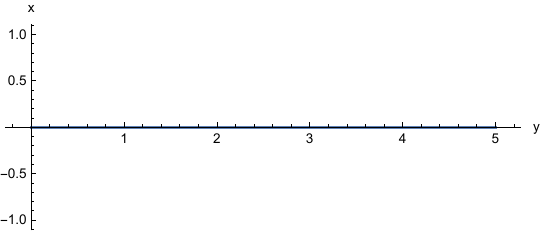}{The curve $(x(t), y(t))$ for $c = 1$, $\p_0 = 0$}{fig:c1p0}{4}
{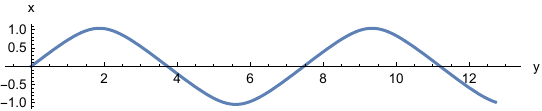}{The curve $(x(t), y(t))$ for $c = 1$, $\p_0 = 1$}{fig:c1p1}{2}

\twofiglabelsizeh
{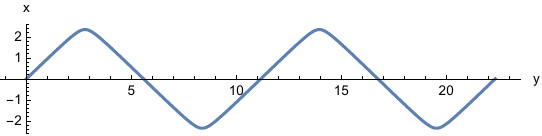}{The curve $(x(t), y(t))$ for $c = 1$, $\p_0 = 2$}{fig:c1p2}{2}
{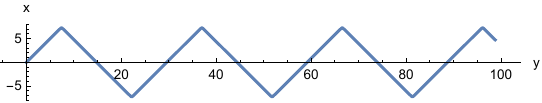}{The curve $(x(t), y(t))$ for $c = 1$, $\p_0 = 4$}{fig:c1p4}{1.8}
}

\subsection{Exponential mapping}
Introduce the exponential mapping
\begin{align*}
&\map{\Exp}{N}{M}, \qquad \Exp(\lam, t) = \pi \circ e^{t \vH}(\lam), \\
&N = \{(\lam, t) \in C \times \R_+ \mid t \in (0, + \infty) \text{ for } c \geq 0; \ \ 
t \in (0, + kK/l) \text{ for } c < 0\}, \\
&C = T_{q_0}^*M \cap \{H = - 1/2, \ h_2 < 0 \}.
\end{align*}
Formulas of Subsec. \ref{subsec:norm_extr} give an explicit parametrization of the exponential mapping.

In this subsection we describe diffeomorphic properties of the exponential mapping via the classical Hadamard's theorem on global diffeomorphism:

\begin{theorem}[\cite{implicit}]\label{th:had}
Let $\map{F}{X}{Y}$ be a smooth mapping of smooth manifolds, $\dim X = \dim Y$. Suppose that the following conditions hold:
\begin{itemize}
\item[$(1)$]
$X$ and $Y$ are connected, 
\item[$(2)$]
$Y$ is simply connected, 
\item[$(3)$]
$F$ is nondegenerate,
\item[$(4)$]
$F$ is proper (i.e., for any compact set $K \subset Y$, the preimage $F^{-1}(K) \subset X$ is compact).
\end{itemize}
Then $F$ is a diffeomorphism.
\end{theorem}

Consider the following stratification in the image of the exponential mapping:
\begin{align*}
&\intt \B_1 = \cup_{i=0}^6 M_i, \\
&M_0 \ : \ x = 0, \qquad -1/24 < \eta < 0, \\
&M_1 \ : \ x > 0, \qquad \f_5(\xi) < \eta < \f_1(\xi), \\
&M_2 \ : \ x < 0, \qquad \f_5(\xi) < \eta < \f_2(\xi), \\
&M_3 \ : \ x > 0, \qquad \f_3(\xi) < \eta < \f_5(\xi), \\
&M_4 \ : \ x < 0, \qquad \f_2(\xi) < \eta < \f_4(\xi), \\
&M_5 \ : \ x > 0, \qquad  \eta = \f_5(\xi), \\
&M_6 \ : \ x < 0, \qquad  \eta = \f_5(\xi), \\
&\f_5(\xi) = (\xi^2-1)/24,
\end{align*}
see Fig. \ref{fig:decompM}.

\figout{
\onefiglabelsizen{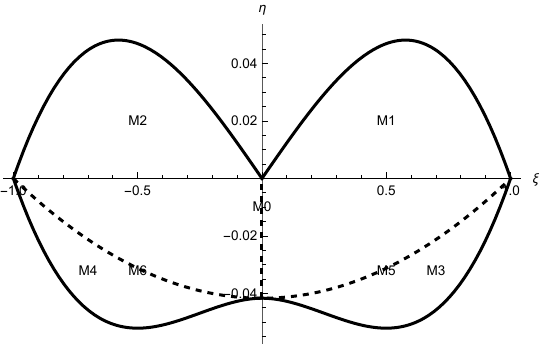}{Stratification in the image of $\Exp$}{fig:decompM}{5}
}

And define the following stratification of the subset
$$
\tN = \{(\lam, t) \in N \mid 
t \in (0, 2 K/(\aen l)) \text{ for } c > 0; \ 
t \in (0, + \infty) \text{ for } c = 0; \ 
t \in (0,  kK/l) \text{ for } c < 0\}
$$
in the preimage of the exponential mapping: 
\begin{align*}
&\tN = \cup_{i=1}^6 \tN_i,\\
&\tN_1 \ : \  c > 0, \qquad \p_0 > 0, \qquad \tau \in (0, 2 K),\\
&\tN_2 \ : \  c > 0, \qquad \p_0 < 0, \qquad \tau \in (0, 2 K),\\
&\tN_3 \ : \  c < 0, \qquad \p_0 > 0, \qquad \tau \in (0,  K),\\
&\tN_4 \ : \  c < 0, \qquad \p_0 < 0, \qquad \tau \in (0,  K),\\
&\tN_5 \ : \  c = 0, \qquad \p_0 > 0,\\
&\tN_6 \ : \  c = 0, \qquad \p_0 < 0.
\end{align*}

\begin{proposition}\label{prop:ExpN1}
There holds the inclusion $\Exp(\tN_1) \subset M_1$. Moreover, the mapping $\map{\Exp}{\tN_1}{M_1}$ is a real-analytic diffeomorphism.
\end{proposition}
\begin{proof}
a) Let us show that $\Exp(\tN_1) \subset M_1$. 

Let $(\lam, t) \in \tN_1$. We have to prove that $(\xi^2-1)/24 < \eta < \xi(1-\xi^2)/8$, i.e., 
$$
(x^2-y^2)y \underset{1)}{<} 24 z-3x^2y-y^3\underset{2)}{<}3x(y^2-x^2).
$$
Inequality $1)$ is rewritten as
\be{1.1}
6z>x^2y.
\ee
Differentiating this inequality by virtue of ODEs \eq{Ham2}, we get
\be{1.2}
x \cosh \p >y \sinh \p.
\ee
Differentiating once more, we get
\be{1.3}
x \sinh \p > y \cosh \p.
\ee
Since $\cosh \p > \sinh \p > 0$, we have $y > x$, and inequality \eq{1.3} is proved. Since the both sides of this inequality vanish at $t = 0$, inequality \eq{1.2} is proved as well. Similarly, inequality \eq{1.1} is proved, thus inequality $1)$ follows.

Now we prove similarly inequality $2)$: it is equivalent to
$$
24z < y^3+3x^2y+3xy^2-3x^3.
$$
Differentiating, we get
$$
3x^2<y^2+2xy,
$$
which holds since $y > x$. Returning back, we prove inequality $2)$.

So the inclusion $\Exp(\tN_1) \subset M_1$ is proved.

\medskip
b) We show that the mapping $\map{\Exp}{\tN_1}{M_1}$ is nondegenerate, i.e., the Jacobian $J = \frac{\partial(x, y, z)}{\partial(c, \p, t)}$ does not vanish on $\tN_1$. 
Direct computation of this Jacobian gives
\begin{align*}
&J = f \cdot \frac{\partial(x, y, z)}{\partial(m, \tau, k)} = f \cdot \frac{9(1-k^2)k^2m^6}{32}J_1, \qquad f \neq 0, \\
&J_1 = J_0 + o(1), \qquad k \to 0,\\
&J_0 = \cos(3\tau)+(8\tau^2-1)\cos\tau - 4 \tau \sin \tau.
\end{align*}
Since
$$
\left(\frac{J_0}{\sin \tau}\right)' = - \frac{2(2\tau-\sin 2 \tau)^2}{\sin^2 \tau} < 0, \qquad \tau \in (0, \pi),
$$
thus $J_0(\tau) < 0$, $\tau \in (0, \pi]$. 

Moreover,
\begin{align}
&\restr{J_1}{\tau=2K} = 12(1-k^2)g(k), \nonumber\\
&g(k) = E^2(k)-2E(k)K(k)+(1-k^2)K^2(k). \label{g(k)}
\end{align}
We have $g(k) = g_1(k) g_2(k)$, $g_1(k) = E(k)-(1+k)K(k)$, $g_2(k) = E(k)-(1-k)K(k)$. Since $g_1'(k) = \frac{E(k)}{k-1} < 0$, then $g_1(k)<0$ for $k \in (0, 1)$. Similarly, since $g_2'(k) = \frac{E(k)}{k+1} > 0$, then $g_2(k)>0$ for $k \in (0, 1)$. Thus $g(k) < 0$ and $\restr{J_2}{\tau = 2 K} < 0$ for $k \in (0, 1)$.

By homotopy invariance of the Maslov index (the number of conjugate points on an extremal)~\cite{cime}, we have $J(\tau) < 0$ for $\tau \in (0, 2 K]$. Thus the mapping $\map{\Exp}{\tN_1}{M_1}$ is nondegenerate.

\medskip
c) We show that this mapping is proper.
In other words, we prove that if $\tN_1 \ni (m, k, \tau) \to \partial \tN_1$, then $q = \Exp(m, k, \tau) \to \partial M_1$. 

1) Let $\tau \to 0$.

1.1) If $m \geq \eps > 0$, then $x \to 0$.

1.2) Let $m \to 0$. 

1.2.1) If $\tau/m \to 0$, then $x \to 0$.

1.2.2) Let $\tau/m \geq \eps > 0$.

1.2.2.1) If $\tau/m \to \infty$, then $y \to \infty$.

1.2.2.2) If $\tau/m \leq K$, then $\f_1 \to 0$.

2) Let $\tau \geq \eps > 0$.

2.1) Let $\am \tau \to \pi$, where $\am \tau$ is the Jacobi amplitude of modulus $k$.

2.1.1) If $m \geq \d > 0$, then $x \to 0$.

2.1.2) If $m \to 0$, then $y \to \infty$.

2.2) Let $\am \tau \leq \pi - \g < \pi$.

2.2.1) Let $m \to 0$.

2.2.1.1) If $k \geq \s > 0$, then $x \to \infty$.

2.2.1.2) Let $k \to 0$.

2.2.1.2.1) If $k/m \to 0$, then $x \to 0$.

2.2.1.2.2) If $k/m \to \infty$, then $x \to \infty$.

2.2.1.2.3) If $\a \leq k/m \leq \b$, then $y \to \infty$.

2.2.2) Let $m \geq \d > 0$.

2.2.2.1) If $m \to \infty$, then $x\to 0$.

2.2.2.2) Let $m \leq \a$.

2.2.2.2.1) If $k \to 0$, then $x \to 0$.

2.2.2.2.2) If $k \to 1$, then $\f_1 \to 0$.

Summing up, the mapping  $\map{\Exp}{\tN_1}{M_1}$  is proper.

Since $\tN_1$ and $M_1$ are connected and simply connected, this mapping is a diffeomorphism by Th.~\ref{th:had}.
\end{proof}

\begin{proposition}\label{prop:ExpN3}
There holds the inclusion $\Exp(\tN_3) \subset M_3$. Moreover, the mapping $\map{\Exp}{\tN_3}{M_3}$ is a real-analytic diffeomorphism.
\end{proposition}
\begin{proof}
Similarly to Propos. \ref{prop:ExpN1}.
\end{proof}

\begin{proposition}\label{prop:ExpN24}
There hold the inclusions $\Exp(\tN_2) \subset M_2$ and $\Exp(\tN_4) \subset M_4$. Moreover, the mappings $\map{\Exp}{\tN_2}{M_2}$ and  $\map{\Exp}{\tN_4}{M_4}$ are real-analytic diffeomorphisms.
\end{proposition}
\begin{proof}
Follows from Propositions \ref{prop:ExpN1} and  \ref{prop:ExpN3} via reflection \eq{refl}.
\end{proof}

\begin{proposition}\label{prop:ExpN5}
There holds the inclusions $\Exp(\tN_5) \subset M_5$, $\Exp(\tN_6) \subset M_6$. Moreover, the mappings $\map{\Exp}{\tN_5}{M_5}$, $\map{\Exp}{\tN_6}{M_6}$ are real-analytic diffeomorphisms.
\end{proposition}
\begin{proof}
Follows from the parametrization \eq{c=0} of the exponential mapping $\restr{\Exp}{\tN5}$.
\end{proof}

Introduce the following subsets in the preimage of the exponential mapping:
$$
N_0^{\pm} = \{(\lam, t) \in N \mid c > 0, \ \sgn \p_0 = \pm 1, \ t = 2K/(\aen l)\}.
$$

\begin{proposition}\label{prop:ExpN0}
There hold the inclusions $\Exp(N_0^{\pm}) \subset M_0$. Moreover, the mappings $\map{\Exp}{N_0^{\pm}}{M_0}$  are real-analytic diffeomorphisms.
\end{proposition}
\begin{proof}
In view of the reflection \eq{refl}, it suffices to consider only the set $N_0^+$. And in view of the dilations \eq{dilat}, it is enough to consider only the case $c = 1$.

If $c = 1$, $\p_0> 0$, $t = 2K/(\aen l) = 2K/k'$, then
\begin{align*}
&x = 0, \\
&y = k'(4E(k)-2k'^2K(k)), \\
&z = \frac 43 k'^3((1+k^2)E(k)-k'^2K(k)),\\
&\xi = 0, \\
&\eta = -\frac{1}{24} + \frac{(1+k^2)E(k)-(1-k^2)K(k)}{6(2E(k)-(1-k^2)K(k))^3} =:\eta_1(k).
\end{align*}
We have $\lim_{k\to 0} \eta_1(k) = - 1/24$, $\lim_{k\to 1} \eta_1(k) = 0$ and $\eta_1'(k) = - \frac{(1-k^2)g(k)}{2k(2E(k)-(1-k^2)K(k))^4}$, where the function $g(k)$ is given by \eq{g(k)}. We showed in the proof of Propos. \ref{prop:ExpN1} that $g(k) < 0$, $k \in (0, 1)$, thus $\map{\eta_1}{(0, 1)}{(-1/24, 0)}$ is a strictly increasing diffeomorphism.

Thus $\Exp(N_0^+) \subset M_0$ and $\map{\Exp}{N_0^+}{M_0}$ is a diffeomorphism.
\end{proof}

Denote 
\begin{align*}
&N_+ = \tN_1 \cup \tN_3 \cup N_5, \\
&N_- = \tN_2 \cup \tN_4 \cup N_6, \\
&M_+ = M_1 \cup M_3 \cup M_5 = (\intt \B_1) \cap \{x > 0\}, \\
&M_- = M_2 \cup M_4 \cup M_6 = (\intt \B_1) \cap \{x < 0\}.
\end{align*}

\begin{theorem}\label{th:ExpN+-}
There hold the inclusions $\Exp(N_{\pm}) \subset M_{\pm}$. Moreover, the mappings $\map{\Exp}{N_{\pm}}{M_{\pm}}$  are real-analytic diffeomorphisms.
\end{theorem}
\begin{proof}
By virtue of the reflection \eq{refl}, it suffices to consider only the domain $N_+$. 

The inclusion $\Exp(N_{+}) \subset M_{+}$ follows from Propositions \ref{prop:ExpN1}, \ref{prop:ExpN3}, \ref{prop:ExpN5}.

The mappings $\restr{\Exp}{\tN_1}$ and $\restr{\Exp}{\tN_3}$ are nondegenerate by Propositions \ref{prop:ExpN1} and  \ref{prop:ExpN3} respectively.
Since $\tN_5 \subset \cl(\tN_1)$ and $\restr{\Exp}{\tN_1}$ is nondegenerate, then it follows by homotopy invariance of the Maslov index that the mapping $\map{\Exp}{N_{+}}{M_{+}}$ is also nondegenerate  at the set $N_5$. Summing up, $\restr{\Exp}{N_+}$ is nondegenerate. 

Similarly to the proof of Propos. \ref{prop:ExpN1} it follows that the mapping $\map{\Exp}{N_{+}}{M_{+}}$ is proper.

Then Theorem \ref{th:had} implies that this mapping is a diffeomorphism.
\end{proof}

\begin{proposition}\label{prop:JN0}
The mapping $\map{\Exp}{N }{M }$ is a local diffeomorphism at points of $N_0^{\pm}$. 
\end{proposition}
\begin{proof}
By virtue of the reflection \eq{refl}, we can consider only the case $N_0^+$.

But $N_0^+ \subset \cl(\tN_1)$ and $\map{\Exp}{\tN_1}{M}$ is nondegenerate. Then by homotopy invariance of the Maslov index the mapping $\map{\Exp}{N }{M }$ is a local diffeomorphism at points of $N_0^{+}$. 
\end{proof}

\subsection{Attainable sets and existence of optimal trajectories}

\begin{theorem}\label{th:att_set}
We have $\A_1 = \B_1$.
\end{theorem}
\begin{proof}
By virtue of Propos. \ref{propos:AinB1}, it remains to prove the inclusion $\A_1 \supset \B_1$. But Theorem \ref{th:ExpN+-} and Propos. \ref{prop:ExpN0} imply that $\Exp(N_+\cup N_-\cup N_0^+) \supset M_+\cup M_-\cup M_0 = \intt \B_1$. Thus $\A_1 \supset \intt \B_1$. 

Further, each point of $\partial \B_1$ is reachable from $q_0$ by an abnormal trajectory, thus $\A_1 \supset \partial \B_1$. 

Summing up, $\A_1 \supset \B_1$, and the equality $\A_1 = \B_1$ follows. 
\end{proof}

Recall that $J^+(q)$ is the causal future of a point $q \in M$, i.e., the attainable set from $q$ for arbitrary nonnegative time. Similarly,  $J^-(q)$ is the causal past of $q$, i.e., the set of points attainable from $q$ for arbitrary nonpositive time. Notice that $\A_1 = J^+({q_0})$. 

\begin{corollary}\label{cor:J+-}
Let $q = (0, y_0, z_0) \in M$. Then
\begin{align*} 
&J^+(q) = 
 \begin{cases}
\f_3(\xi_0) \leq \eta_0 \leq \f_1(\xi_0), \qquad& 0 \leq \xi_0\leq 1, \qquad y \geq y_0,\\
\f_4(\xi_0) \leq \eta_0 \leq \f_2(\xi_0), \qquad& -1 \leq \xi_0\leq 0, \qquad y \geq y_0,
\end{cases}\\
&\xi_0 = \frac{x}{y-y_0}, \qquad \eta_0 = \frac{24(z-z_0)-3x^2(y-y_0)-(y-y_0)^3}{24(y-y_0)^3}, \\
&J^-(q) = 
 \begin{cases}
\f_3(\xi_1) \leq \eta_1 \leq \f_1(\xi_1), \qquad& 0 \leq \xi_1\leq 1, \qquad y \leq  y_0,\\
\f_4(\xi_1) \leq \eta_1 \leq \f_2(\xi_1), \qquad& -1 \leq \xi_1\leq 0, \qquad y \leq  y_0,
\end{cases}\\
&\xi_1 = -\xi_0, \qquad \eta_1 = \eta_0.
\end{align*}
\end{corollary}
\begin{proof}
The expression for $J^+(q)$ follows from the expression for $J^+({q_0}) = \A_1$ via the translation \eq{paral}.
And the expression for $J^-(q)$ follows by time inversion. 
\end{proof}

\begin{theorem}\label{th:exist}
Let points $q_0, q_1 \in M$ satisfy the inclusion $q_1 \in J^+({q_0})$. Then there exists an optimal trajectory in problem \eq{pr11}--\eq{pr14}.
\end{theorem}
\begin{proof}
By Theorem 2 and Remark 2 \cite{SL_exist}, the following conditions are sufficient for existence of an optimal trajectory connecting $q_0$ and $q_1$:
\begin{itemize}
\item[$(1)$]
$q_1 \in J^+({q_0})$,
\item[$(2)$]
There exists a compact $K \subset M$ such that $J^+({q_0}) \cap J^-_{q_1} \subset K$, 
\item[$(3)$]
$T(q_0, q_1) < + \infty$,
\end{itemize}
where $T(q_0, q_1)$ is the supremum of time required to reach $q_1$ from $q_0$ for a trajectory of system \eq{pr11} reparametrized so that $u_1 \equiv 1$.

Condition (1) holds by assumption of this theorem.

Condition (3) holds since if $u\equiv 1$ then $\dot y \equiv 1$, thus $T(q_0, q_1) = y_1-y_0 < + \infty$.

Now we prove condition (2). It is easy to see from Cor. \ref{cor:J+-} that
$$
\cup_{q \in \Pi} J^+(q) = \cup_{q \in \Pi} J^-(q) = M.
$$
Thus there exist $p_0, p_1 \in \Pi$ such that $p_0 \in J^+({q_0})$, $q_1 \in J^-({p_1})$. Thus $J^+({q_0}) \subset J^+({p_0})$, $J^-({q_1}) \subset J^-({p_1})$, so $J^+({q_0})\cap J^-({q_1}) \subset J^+({p_0})\cap J^-({p_1})$. It is easy to see from Cor. \ref{cor:J+-} that the set $K:= J^+({p_0})\cap J^-({p_1})$ is compact. So condition (2) above holds, and this theorem is proved.
\end{proof}

\subsection{Optimality of extremal trajectories}
%\subsubsection{Necessary optimality conditions}
\begin{proposition}\label{prop:max}
Let $\lam = (\p_0, c) \in C$, $c > 0$, $\p_0 \neq 0$, and let $t_1 > 2K/({\text{\em  \ae}} l)$. Then the extremal trajectory $q(t) = \Exp(\lam, t)$, $t \in [0, t_1]$, is not optimal.
\end{proposition}
\begin{proof}
By virtue of the reflection \eq{refl}, we can assume that $\p_0>0$. By contradiction, suppose that $q(t)$, $t \in [0, t_1]$, is optimal. 

Let $\tlam = (-\p_0, c) \in C$, $\tq(t) = \Exp(\tlam, t)$, $t \in [0, t_1]$. Denote $\bt = 2K/(\aen l)$, then $q(\bt) = \tq(\bt)$ and 
$l(\restr{q(\cdot)}{[0, \bt]}) = l(\restr{\tq(\cdot)}{[0, \bt]})$, i.e., the trajectories $q(\cdot)$ and $\tq(\cdot)$ have a Maxwell point at $t = \bt$ \cite{max1}. Now consider the trajectory
$$
\hq(t) = \begin{cases}
\tq(t), \qquad &t \in [0, \bt], \\ 
q(t), \qquad &t \in [\bt, t_1].
\end{cases}
$$
Since $l(\restr{\hq(\cdot)}{[0, t_1]}) = l(\restr{q(\cdot)}{[0, t_1]})$, then the trajectory $\hq(t)$, $t \in [0, t_1]$, is optimal. But $\hq(t)$ has a corner point at $t = \bt$, which is not possible for normal trajectories. Thus $\hq(t)$ is abnormal, so its support belongs to $\partial \A_1$. But this is impossible since the support of the normal trajectory $q(t)$ belongs to $\intt \A_1$. A contradiction obtained completes the proof.
\end{proof}

Define the following function:
\begin{align*}
&\map{\tt}{C}{(0, + \infty]}, \qquad \lam = (\p_0, c) \in C, \\
&c = 0 \then \tt(\lam) = + \infty, \\
&c \neq 0, \ \p_0 = 0 \then \tt(\lam) = + \infty, \\
&c >0, \ \p_0 \neq 0 \then \tt(\lam) = \frac{2kK}{l}, \\
&c <0, \ \p_0 \neq 0 \then \tt(\lam) = \frac{kK}{l}.
\end{align*}
We prove (see Cor. \ref{cor:cut}) that $\tt(\lam)$ is the cut time for an extremal trajectory $\Exp(\lam, t)$:
$$
\tt(\lam) = \tcut(\lam) := \sup \{t_1 > 0 \mid \Exp(\lam, t) \text{ is optimal for } t \in [0, t_1]\}.
$$

\begin{theorem}\label{th:opt1}
Let $\lam \in C$, $t_1 \in (0, \tt(\lam))$. Then the trajectory $q(t) = \Exp(\lam, t)$, $t \in [0, t_1]$, is optimal.
\end{theorem}
\begin{proof}
By Theorem \ref{th:exist}, there exists an optimal trajectory connecting $q_0$ and $q_1 := q(t_1)$. Since $q_1 \in \intt \A_1$, this trajectory is a normal extremal trajectory. Since $q_1 \in M_+\cup M_-$, by Theorem \ref{th:ExpN+-} there exists a unique arclength parametrized normal extremal trajectory connecting $q_0$ and $q_1$, so it coincides with $q(t)$.
\end{proof}

\begin{theorem}\label{th:opt2}
Let $\lam = (\p_0, c)\in C$, $c > 0$, $\p_0 \neq 0$, $t_1 =\tt(\lam)$. Then the trajectory $q(t) = \Exp(\lam, t)$, $t \in [0, t_1]$, is optimal.
\end{theorem}
\begin{proof}
Similarly to Th. \ref{th:opt1} with the only distinction that now there are exactly two optimal trajectories corresponding to the covectors $(\pm \p_0, c) \in C$, these trajectories are symmetric by virtue of the reflection \eq{refl}. 
\end{proof}

Proposition \ref{prop:max} and 
Theorems \ref{th:opt1}, \ref{th:opt2} imply the following.

\begin{corollary}\label{cor:cut}
For any $\lam \in C$ we have $\tcut(\lam) = \tt(\lam)$.
\end{corollary}

\begin{proposition}\label{prop:opt3}
Any abnormal extremal trajectory is optimal.
\end{proposition}
\begin{proof}
Abnormal extremal trajectories are exactly trajectories belonging to the boundary of the attainable set $\A_1$. If $q(t)$, $t \in [0, t_1]$, is an abnormal trajectory, then, up to reparametrization, the corresponding control has the following form:
\begin{align*}
&1) \ u(t) = 
\begin{cases}
(\pm 1, 1), & t \in [0, \tau_1], \\
(\mp 1, 1), & t \in [\tau_1, t_1],
\end{cases}
\quad \tau_1 \in [0, t_1], \text{ or}\\ 
&2) \ u(t) = 
\begin{cases}
(0, 1), & t \in [0, \tau_1], \\
(\pm 1, 1), & t \in [\tau_1, t_1],
\end{cases}
\quad \tau_1 \in [0, t_1]. 
\end{align*} 
In the case 1) we have $l(q(\cdot)) = 0$. If $\tq(t)$, $t \in [0, \tilde{t}_1]$, is a trajectory such that $\tq(\tilde{t}_1) = q(t_1)$, then, up to reparametrization, $\tq(t) \equiv q(t)$, thus $q(t)$ is optimal.

In the case 2) we have $l(q(\cdot)) = \tau_1$, and a similar argument shows that $q(t)$ is optimal.
\end{proof}

Theorems \ref{th:opt1}, \ref{th:opt2} and Propos. \ref{prop:opt3} yield the following description of the optimal synthesis.

\begin{theorem}\label{th:synth}
\begin{itemize}
\item[$(1)$]
Let $q_1 \in (\intt \A_1) \setminus \Pi$. Then there exists a unique optimal trajectory $\Exp(\lam, t)$, $t \in [0, t_1]$, where $(\lam, t_1) = \Exp^{-1}(q_1)$. 
\item[$(2)$]
Let $q_1 \in (\intt \A_1) \cap \Pi$. Then there exist exactly two optimal trajectories $\Exp(\lam_i, t)$, $t \in [0, t_1]$, $i = 1, 2$,
where $\{(\lam_1, t_1), (\lam_2, t_1)\} = \Exp^{-1}(q_1)$. 
\item[$(3)$]
Let $q_1 = (x_1, y_1, z_1) \in S_1 \cup S_1$, $\sgn x_1 = \pm 1$. Then there exists a unique optimal trajectory
$$
q(t) = \begin{cases}
e^{t(\pm X_1+X_2)}(q_0), \qquad &t \in [0, \tau_1], \\
e^{(t-\tau_1)(\mp X_1+X_2)} \circ e^{\tau_1(\pm X_1+X_2)}(q_0), \qquad &t \in [\tau_1, t_1], 
\end{cases}
$$
$\tau_1 = (y_1\pm x_1)/2$, $t_1 = y_1$.
\item[$(4)$]
Let $q_1 = (x_1, y_1, z_1) \in S_3 \cup S_4$, $\sgn x_1 = \pm 1$. Then there exists a unique optimal trajectory
$$
q(t) = \begin{cases}
e^{t X_2}(q_0), \qquad &t \in [0, \tau_1], \\
e^{(t-\tau_1)(\pm X_1+X_2)} \circ e^{\tau_1 X_2}(q_0), \qquad &t \in [\tau_1, t_1], 
\end{cases}
$$
$\tau_1 = y_1\pm x_1$, $t_1 = y_1$.
\item[$(5)$]
Let $q_1 = (0, y_1, 0)$, $y_1>0$. Then there exists a unique optimal trajectory
$
q(t) =  
e^{t X_2}(q_0)$, $t \in [0, t_1]$,
 $t_1 = y_1$.
\item[$(6)$]
Let $q_1 = (0, y_1, z_1)$, $z_1 = y_1^3/24>0$. Then there exist exactly two optimal trajectories
\begin{align*}
&q^1(t) = \begin{cases}
e^{t(X_1+X_2)}(q_0), \qquad &t \in [0, \tau_1], \\
e^{(t-\tau_1)(-X_1+X_2)} \circ e^{\tau_1(X_1+X_2)}(q_0), \qquad &t \in [\tau_1, t_1], 
\end{cases}\\
&q^2(t) = \begin{cases}
e^{t(-X_1+X_2)}(q_0), \qquad &t \in [0, \tau_1], \\
e^{(t-\tau_1)(X_1+X_2)} \circ e^{\tau_1(-X_1+X_2)}(q_0), \qquad &t \in [\tau_1, t_1], 
\end{cases}
\end{align*}
$\tau_1 = t_1/2$, $t_1 = y_1$.
\item[$(7)$]
Let $q_1 = (x_1, y_1, z_1)$, $\sgn x_1 = \pm 1$, $y_1 = \pm x_1$, $z_1 = y_1^3/6$. Then there exists a unique optimal trajectory
$
q(t) =  
e^{t (\pm X_1 + X_2)}(q_0)$, $t \in [0, t_1]$,
 $t_1 = y_1$.
\end{itemize}
\end{theorem}

\begin{remark}
In Theorem $\ref{th:synth}$ existence of exactly one (two) optimal trajectories is understood up to reparametrization.
\end{remark}

\subsection{\SL distance}
In this subsection we study the function $d(q) := d(q_0, q)$, $q \in M$.

\begin{theorem}\label{th:dist}
\begin{itemize}
\item[$(1)$]
The distance $d$ is real-analytic on $(\intt \A_1) \setminus \Pi$ and continuous on $\intt \A_1$. 
\item[$(2)$]
The distance $d$ has discontinuity of the first kind at each point of $S_3 \cup S_4$.
\item[$(3)$]
The restriction $\restr{d}{\Pi}$ is real-analytic on the set $\{x=0, \ z \in (0, y^3/24), \ y> 0\}$ and discontinuous of the first kind on the set $\{x=0, \ z = y^3/24, \ y> 0\}$.
\item[$(4)$]
The distance $d$ is homogeneous of order $1$ w.r.t. dilations \eq{dilat}:
$$
d(\d_{\a}(q)) = \a \d(q), \qquad \d_{\a}(x, y, z) = (\a x, \a y, \a^3 z), \qquad \a > 0, \quad q = (x, y, z) \in M.
$$
\end{itemize}
\end{theorem}
\begin{proof}
(1) Follows from Th. \ref{th:ExpN+-} and Propos. \ref{prop:JN0}.

(2) Take any $q_1 = (x_1, y_1, z_1) \in S_3 \cup S_4$, then $d(q_1) = y_1 >0$ by item (4) of Th. \ref{th:synth}. On the other hand, $q_1 \in \partial \A_1$, thus there are points $q \in M \setminus \A_1$ arbitrarily close to $q_1$. Since $d(q)=0$, the distance $d$ has discontinuity of the first kind at $q_1$. 

(3) The restriction $\restr{d}{\Pi}$ is real-analytic on the set $\{x=0, \ z \in (0, y^3/24), \ y> 0\}$ by Propos. \ref{prop:JN0}.

Let $q_1 = (x_1, y_1, z_1) \in \A_1$, $x_1=z_1=0$, $y_1 > 0$.
Then $q_1 \in \partial \A_1$, and the distance $d$ has discontinuity of the first kind at $q_1$ similarly to item (2).

(4) is obvious in view of symmetry \eq{dilat}.
\end{proof}

Let $q = (0, y, z) \in \Pi$, then $d(q) = y^3 d(0, 1, z/y^3) = y^3d(0, 1, \eta + 1/24)$ since $\restr{\eta}{\Pi} = z/y^3-1/24$. We plot the function $\eta \mapsto d(0, 1, \eta + 1/24)$ in Fig. \ref{fig:d0}.

\figout{
\onefiglabelsizen{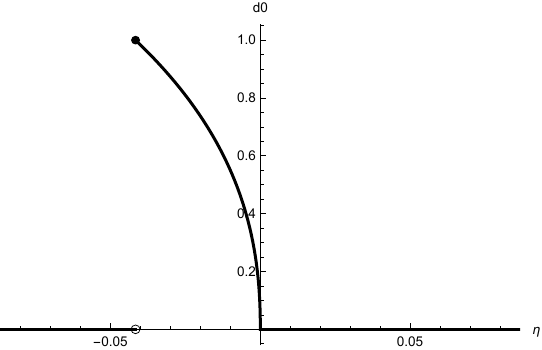}{Plot of $\eta \mapsto d(0, 1, \eta + 1/24)$}{fig:d0}{5}
}

\subsection{\SL sphere}
By virtue of dilations \eq{dilat}, the Lorentzian spheres
$$
S(R) = \{q \in M \mid d(q) = R\}, \qquad R > 0,
$$
satisfy the relation $S(R) = \d_R(S(1))$, thus we describe only the unit sphere $S := S(1)$.

\begin{theorem}\label{th:sphere}
\begin{itemize}
\item[$(1)$]
The set $S \setminus (\Pi \cup \partial \A_1)$ is a real-analytic manifold.
\item[$(2)$]
The sphere $S$ is nonsmooth and Lipschitzian at points of $S \cap \Pi$.
\item[$(3)$]
The intersection $S \cap \Pi$ is defined parametrically for $k \in [0, 1)$:
\begin{align*}
&y = \frac{4 E(k) - 2(1-k^2)K(k)}{2(1-k^2)K(k)} = 1 + k^2 + O(k^4), \qquad k \to 0,\\
&z = \frac{2(1-k^2)K(k) -2(1+k^2) E(k)}{12(1-k^2)^3K^3(k)} = \frac{k^2}{\pi^2} + O(k^4), \qquad k \to 0.
\end{align*}
In particular, the intersection $S \cap \Pi$ is semi-analytic.
\end{itemize}
\end{theorem}
\begin{proof}
(1) Follows from Th. \ref{th:ExpN+-}.

(2) The hemi-spheres $S_{\pm} = S \cap \{ \pm x \geq 0\}$ have at points of $S \cap \Pi$ the normal vectors $n_{\pm} = \pm(g(k), *, *)$, where the function $g(k)$ is given by \eq{g(k)} and is negative for $k \in (0, 1)$.  So the sphere~$S$ has a transverse self-intersection at points of $S \cap \Pi$. .

(3) The parametrization of $S \cap \Pi$ is obtained from the parametrization of the exponential mapping in the case 2.2 of Subsec. \ref{subsec:norm_extr} for $\tau = 2 K$, $t = 1$. 
\end{proof}

\begin{remark}
One of the most important results of the paper \cite{ABCK} is non-subanalyticity of the sub-Riemannian sphere in the Martinet flat case. Its proof relies on non-semianalitycity of the intersection of the sub-Riemannian sphere with the Martinet surface. Item (3) of the preceding theorem states that such an intersection is semi-analytic in the sub-Lorentzian case. We leave the question of subanalyticity of the sub-Lorentzian sphere in the first flat problem on the Martinet  distribution open since we cannot conclude on sub-analitycity (or its lack) at points of the boundary of the sphere~$S$.
\end{remark}

Denote $S_n = S \cap \intt \A_1$, $S_a = S \cap \partial \A_1$.

\begin{remark}
The set $S_n$ (resp. $S_a$) is filled with the endpoints of optimal normal (resp. abnormal) trajectories of length $1$ starting at $q_0$.
\end{remark}

\begin{lemma}\label{lem:SnSa}
We have $\cl(S_n) \supset S_a$.
\end{lemma}
\begin{proof}
Take any point $q \in S_a$ and choose any neighbourhood $q \in O \subset M$. By Proposition 8.1 \cite{groch2}, the distance $d$ is upper semicontinuous on $\A_1$, thus there exists a point $q_- \in O \cap \intt \A_1$ such that $d(q_-) \leq d(q) = 1$. Further, for small $t > 0$ there exists a point 
$$
q_+ = e^{t(u_1X_1+u_2X_2)}(q) \in O \cap \intt \A_1, \qquad (u_1, u_2) \in U_1.
$$
Then $d(q_+)\geq 1$. Since $O \cap \intt \A_1$ is arcwise connected and $\restr{d}{\intt \A_1}$ is continuous, there exists a point $\bq \in O \cap \intt \A_1$ such that $d(\bq) = 1$. Then $\bq \in S_n \cap O$. So $\cl(S_n) \supset S_a$. 
\end{proof}

\begin{comment}
\begin{lemma}
For any point $q \in S_a$ and any neighbourhood $q \in O \subset M$ there exists a triple $\{q^+, \ q^-, \ q^0\} \subset \intt \A_1 \cap O$ such that $d(q^+) > 1$, $d(q^-) > 1$, $d(q^0) > 1$.
\end{lemma}
\begin{proof}
We can assume that $q = e^{t(X_1+X_2)}\circ e^{X_2}(q_0)$, $t \geq 0$, and let $q \in O \subset M$ be any neighbourhood of $q$. Let $q_{\pm} := e^{t(X_1+X_2)}\circ e^{(1\pm \eps)X_2}(q_0) \in \partial \A_1$, where $\eps > 0$ is chosen small so that $q_{\pm} \in O$. Notice that $d(q_{\pm}) = 1 \pm \eps$, i.e., $q_{\pm} \in S(1 \pm \eps) \cap \partial \A_1$.

By virtue of Lem. \ref{lem:SnSa} and dilation \eq{dilat}, we have $\cl(S(1\pm\eps) \cap \intt \A_1) \supset S(1\pm\eps)\cap \partial \A_1$. Thus there exist points $q^{\pm} \in  S(1\pm\eps)\cap \partial \A_1\cap O$. Then $d(q^{\pm}) = 1 \pm \eps$. Since $\restr{d}{\intt \A_1}$ is continuous, there exists a point $q^0 \in \intt \A_1 \cap O$ such that $d(q^0) = 1$.
\end{proof}
\end{comment}

\begin{proposition}\label{propos:homeo}
The sphere $S$ is   homeomorphic to the closed half-plane $\R^2_+ := \{(a, b) \in \R^2 \mid b \geq 0\}$.
\end{proposition}
\begin{proof}
Denote the group of dilations \eq{dilat} as $G = \{\d_{\a} \mid \a > 0\}$ and consider the projection
\begin{align}
&\map{p}{\Mc}{\Mc /G}, \qquad \Mc = M \setminus \{q_0\}.\label{p}
\end{align}
If $y \neq 0$ (which is the case on $S$), then projection \eq{p} is given in coordinates as
$$
p \ : \ q = (x, y, z) \mapsto \s = (\xi, \eta),
$$
where $\xi$, $\eta$ are defined in \eq{xieta}. The mapping \eq{p} is smooth.

The image of the sphere $S$ under the action of the projection $p$ is given as follows:
$$
p(S) \ : \ 
\begin{cases}
\f_3(\xi) \leq \eta < \f_1(\xi), & 0 < \xi < 1, \\
\f_4(\xi) \leq \eta < \f_2(\xi), & -1 < \xi < 0,
\end{cases}
$$
see Fig. \ref{fig:p(S)}. It is obvious that $p(S)$ is homeomorphic to $\R^2_+$. Let us show that $\map{p}{S}{p(S)}$ is a homeomorphism.

First, the mapping $\map{p}{S}{p(S)}$ is a bijection since the sphere $S$ intersects with each orbit of~$G$ at not more than one point.

Second, the mapping $\map{p}{S}{p(S)}$ is continuous as a restriction of a smooth mapping \eq{p}.

It remains to prove that the inverse mapping $\map{p^{-1}}{p(S)}{S}$ is continuous. Continuity of $\restr{p^{-1}}{p(S_n)}$ and 
$\restr{p^{-1}}{p(S_a)}$ is obvious. Let
$$
p(S_n) \ni \s_k \underset{k \to \infty}{\to} \bs \in p(S_a),
$$
we show that $p^{-1}(\s_k) =: q_k \underset{k \to \infty}{\to} \bq := p^{-1}(\bs)$.
It follows from Propos. \ref{prop:ExpN3} that $q_k \underset{k \to \infty}{\to}\hq \in S_a$, where $p(\hq) = \bs$. Since orbits of $G$ are one-dimensional, it remains to prove that $d(\hq) = 1$. 

By Lemma \ref{lem:SnSa}, there exists a sequence $S_n \ni \tq_k \underset{k \to \infty}{\to} \bq$. Thus $\ts_k := p(\tq_k) \underset{k \to \infty}{\to} \bs$. Denote by $\rho$ and~$\widehat{\rho}$ the Euclidean distances in $M$ and $\Mc /G$ respectively. Then $\widehat{\rho}(\s_k, \ts_k) \underset{k \to \infty}{\to} 0$, thus ${\rho}(q_k, \tq_k) \underset{k \to \infty}{\to} 0$, whence $\hq = \bq$.

Thus $\map{p}{S}{p(S)}$ is a homeomorphism, and $S \simeq \R^2_+$.
\end{proof}

\begin{corollary}
\begin{itemize}
\item[$(1)$]
The sphere $S$ is a topological manifold with boundary $S_a$, homeomorphic to the closed half-plane $\R^2_+$.
\item[$(2)$]
The sphere $S$ is a stratified space with real-analytic strata $S_n^{\pm} := S_n \cap \{x \gtrless 0\}$, $S_n^0 := S \cap \Pi$, $S_a^ {\pm} := S_a \cap \{x \gtrless 0\}$, $S_a^0 := S_a \cap \Pi = \{(0, 1, 0)\}$. Wherein
there are diffeomorphisms
$S_n^{\pm} \to \R^2$; $S_n^0, S_a^{\pm} \to \R$.
\item[$(3)$]
Under a homeomorphic embedding of the sphere $S$ into the half-plane $\R^2_+= \{(a, b) \in \R^2 \mid b \geq 0\}$ the indicated strata are mapped as follows:
$S_n^{\pm} \to \{a \gtrless 0, b > 0\}$,
$S_n^{0} \to \{a = 0, b > 0\}$,
$S_a^{\pm} \to \{a \gtrless 0, b = 0\}$,
$S_a^0 \to \{a=b=0\}$.
\end{itemize}
\end{corollary}

See Fig. \ref{fig:p(S)}.

\begin{remark}
There is a numerical evidence that the sphere $S$ is a piecewise smooth manifold with boundary, with a stratification shown in Fig. $\ref{fig:p(S)}$.
\end{remark}

\figout{
\twofiglabelsizeh
{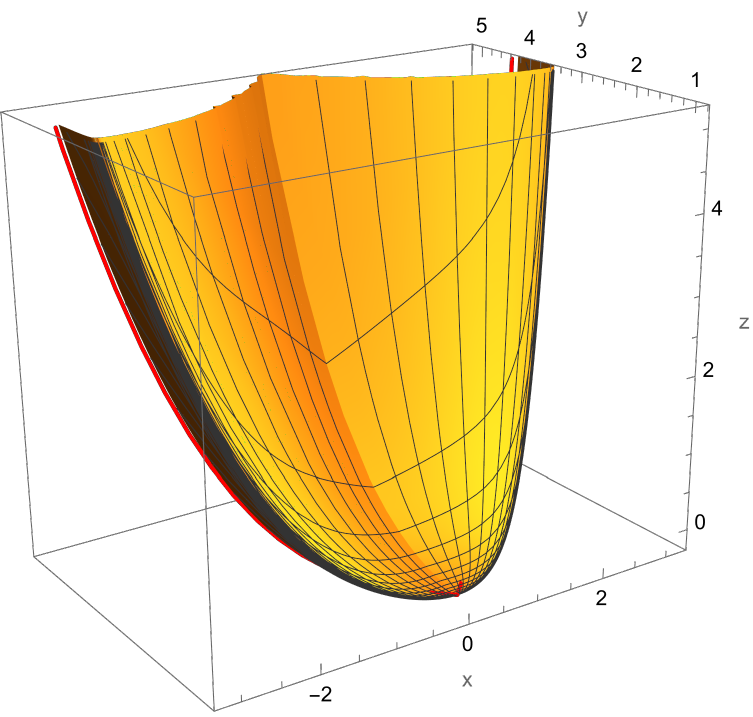}{The sphere $S$}{fig:s1}{6}
{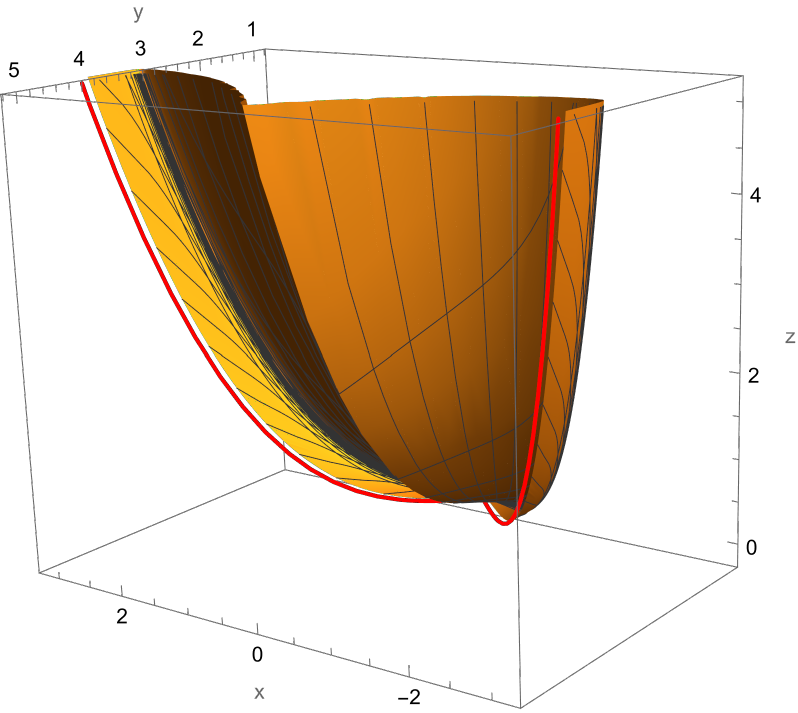}{The sphere $S$}{fig:s2}{6}

\twofiglabelsizeh
{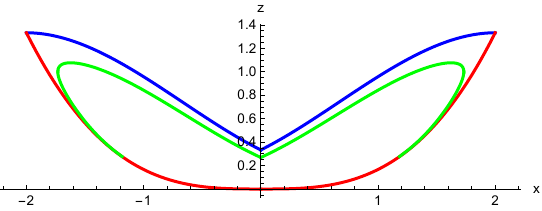}{Intersection of $S$ and $\partial \A_1$ with $\{y=2\}$}{fig:sphy2}{4.2}
{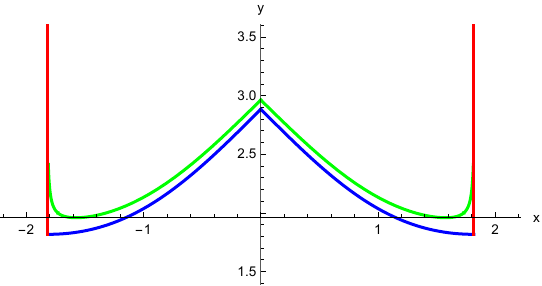}{Intersection of $S$ and $\partial \A_1$ with $\{z=1\}$}{fig:sphz1}{4.2}

\twofiglabelsizeh
{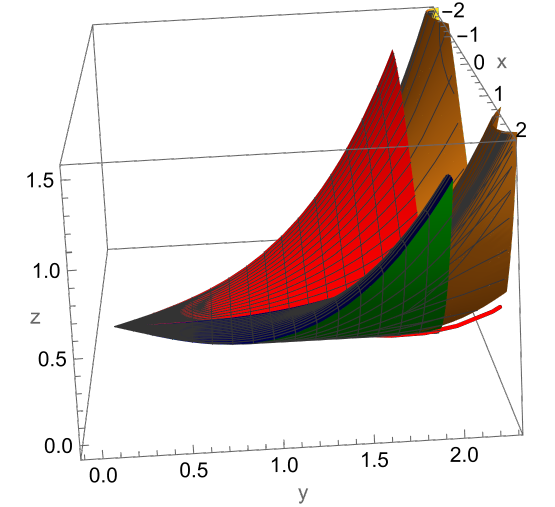}{  $S$ inside of $\partial \A_1$}{fig:SB1_1}{6}
{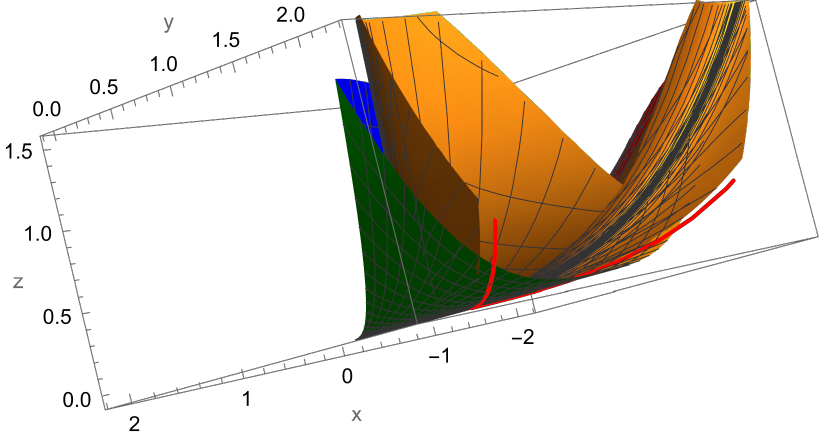}{  $S$ inside of $\partial \A_1$}{fig:SB1_2}{6}

\twofiglabelsizeh
{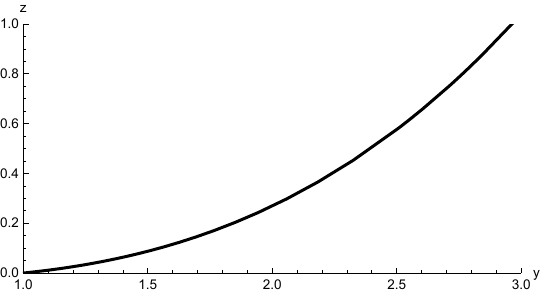}{Intersection of $S$   with $\Pi$}{fig:s0}{4}
{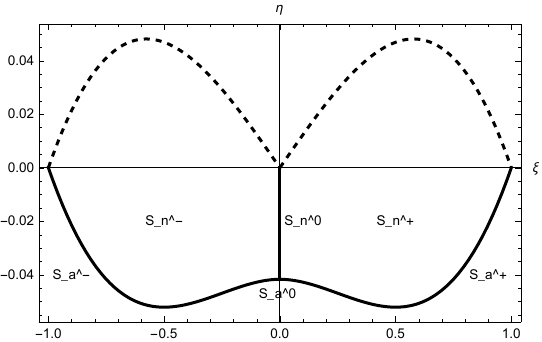}{Stratification of $p(S)$}{fig:p(S)}{6}
}

\begin{proposition}
The sphere $S$ is closed.
\end{proposition}
\begin{proof}
Since $S = S_n \cup S_a$, $\cl(S_n) \supset S_a$ and $\cl(S_a) = S_a$, we have to prove that $\cl(S_n) \subset S$. Take any sequence $S_n \ni q_k \underset{k \to \infty}{\to} \bq$. Denote $\s_k = p(q_k), \ \bs = p(\bq) \in \Mc /G$. We may consider only the case $\bs \in p(S_a)$. If $\bs = (\bar{\xi},\bar{\eta})$ is on the upper part of the boundary of $p(S_n)$, i.e., $\bar{\eta} = \f_1(\bar{\xi})$, $\bar{\xi} \in (0, 1)$, or $\bar{\eta} = \f_2(\bar{\xi})$, $\bar{\xi} \in (0, 1)$, then it follows from the proof of Propos. \ref{prop:ExpN1}  that $q_k \underset{k \to \infty}{\to}  \infty$, which is impossible. And if $\bs = (\bar{\xi},\bar{\eta})$ is on the lower part of the boundary of $p(S_n)$, i.e., $\bar{\eta} = \f_3(\bar{\xi})$, $\bar{\xi} \in (0, 1)$, or $\bar{\eta} = \f_4(\bar{\xi})$, $\bar{\xi} \in (0, 1)$, then $\bq \in \partial \A_1$. Similarly to the proof of Propos. \ref{propos:homeo} it follows that $d(\bq) = 1$. So the claim of this proposition follows.
\end{proof}

\begin{proposition}
The restriction $\restr{d}{\A_1}$ is continuous.
\end{proposition}
\begin{proof}
Let $\A_1 \ni q_k \underset{k \to \infty}{\to} \bq \in \A_1$, we have to prove that $d(q_k) \underset{k \to \infty}{\to} d(\bq)$. 

If $\bq \in \intt \A_1$, then the claim follows by Th. \ref{th:dist}.

Let $\bq \in \partial \A_1$. Since $d$ is upper semicontinuous (Propos. 8.1 \cite{groch2}), we have $\limsup_{k \to \infty} d(q_k) \leq d(\bq)$. In order to show that $d(\bq) \leq \liminf_{k\to \infty} d(q_k)$, we assume by contradiction that $d(\bq) > \liminf_{k\to \infty} d(q_k)$. In other words, there exists a subsequence $\{q_{k_m}\}$ such that $d(\bq) > \lim_{m \to \infty} d(q_{k_m})$. By virtue of dilation \eq{dilat}, we can construct a sequence $q^m \in S$ converging to a point $\widehat q \in \partial \A_1$ with $ d(\widehat q) > 1$. This contradicts to closedness of $S$. Thus $d(q_k) \underset{k \to \infty}{\to} d(\bq)$, and the statement follows.
\end{proof}

\section{The second problem}\label{sec:P2}  
In this section we consider a flat \sL problem on the  Martinet  distribution whose future cone has trivial intersection with the tangent plane to the Martinet surface $\Pi$. It is natural that this problem is more simple than the first problem considered in the previous section. All proofs for the second problem are completely similar or more simple than for the first one, so we skip them.

\subsection{Problem statement}

The second flat sub-Lorentzian problem on the Martinet distribution is stated as the following optimal control problem:
\begin{align}
&\dot q = u_1 X_1 + u_2 X_2, \qquad q \in M, \label{pr21} \\
&u= (u_1, u_2) \in U_2 = \{u_1 \geq |u_2|\}, \label{pr22}\\
&q(0) = q_0 = (0, 0, 0), \qquad q(t_1) = q_1, \label{pr23}\\
&l = \int_0^{t_1} \sqrt{u_1^2-u_2^2}dt \to \max. \label{pr24}
\end{align}
See Fig. \ref{fig:U2}.

\subsection{Invariant set}
By PMP, the boundary of the attainable set $\A_2$ of system \eq{pr21}, \eq{pr22} from the point $q_0$ for arbitrary nonnegative time consists of lightlike trajectories corresponding to piecewise constant controls with values $u = (1, \pm 1)$ and up to one switching. These trajectories fill the boundary of the set
$$
\B_2 = 
\{(x, y, z) \in M \mid x \geq |y|, \ z^1(x, y) \leq z \leq z^2(x,y)\},
$$
where $z^1(x, y) = ((x+y)^3-4x^3)/24$, $z^2(x, y) = (4x^3-(x-y)^3)/24$. See Fig. \ref{fig:B2}.

\begin{proposition}\label{propos:AinB2}
The set $\B_2$ is an invariant domain of system \eq{pr21}, \eq{pr22}. Moreover, $\A_2 \subset \B_2$.
\end{proposition} 
\begin{proof}
Similarly to Propos. \ref{propos:AinB1}.
\end{proof}

\figout{
\twofiglabelsizeh
{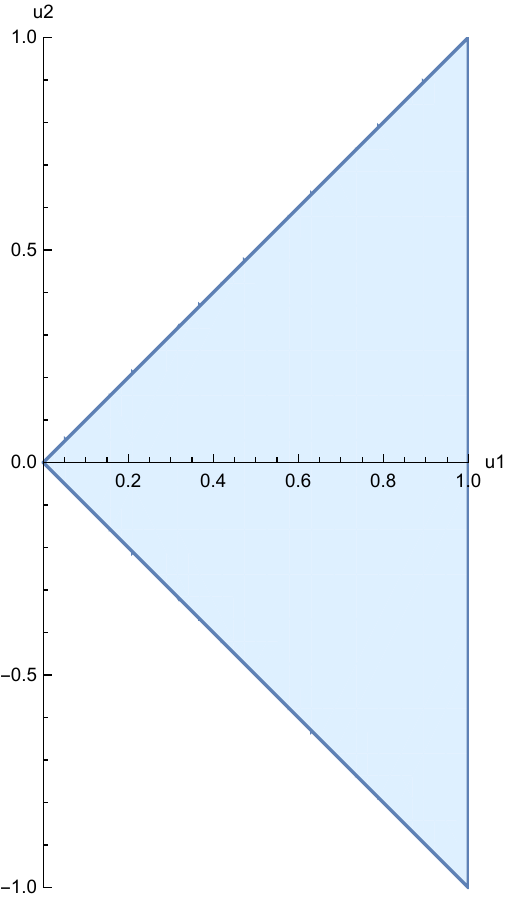}{The set $U_2$}{fig:U2}{6}
{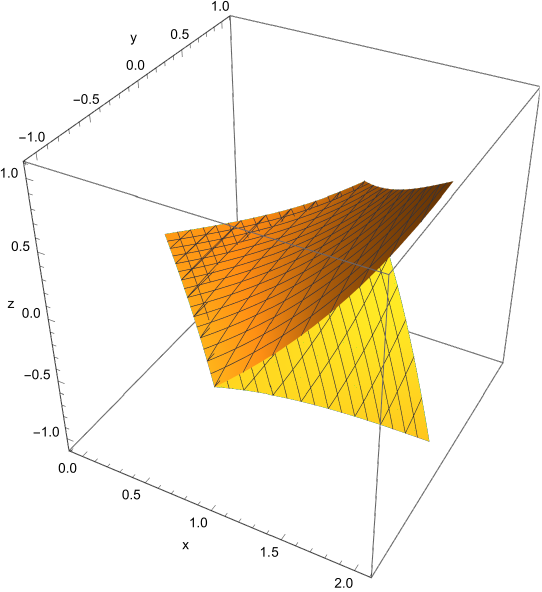}{The set $\B_2$}{fig:B2}{7}
}

\subsection{Extremal trajectories}
\subsubsection{Abnormal extremal trajectories}
Abnormal   trajectories, up to time reparametrization, correspond to controls $u = (1, \pm 1)$ with up to one switching.

\subsubsection{Normal extremals}\label{subsec:norm_extr2}
Normal
 extremals satisfy the Hamiltonian system with the Hamiltonian $H = (-h_1^2+h_2^2)/2$, $h_1 < - |h_2|$:
\be{Ham20}
\dh_1 = -h_2h_3x, \quad \dh_2 = -h_1h_3x, \quad \dh_3 = 0, \quad \dot x = -h_1, \quad \dot y =  h_2, \quad \dot z =  h_2 x^2/2.
\ee
We can choose arclength parameterization on normal extremal trajectories and thus assume that $H \equiv -1/2$. In the coordinates $h_1 = -\cosh \p$, $h_2 =  \sinh \p$, $h_3 = c$; $\p, c \in \R$, the Hamiltonian system \eq{Ham20} reads
\begin{align}
&\dot \p =  cx, \qquad \dot c = 0, \label{Ham21}\\
&\dot x = \cosh \p, \qquad \dot y = \sinh \p, \qquad \dot z = \sinh \p x^2/2. \label{Ham22}
\end{align}
This system has a first integral $E = \frac{cx^2}{2} - \sinh \p \in \R$.

Solutions to this system with the initial condition $\p(0) = \p_0$, $x(0) = y(0) = z(0) = 0$ are as follows.

1) If $c = 0$, then
$$
\p \equiv \p_0, \qquad x = t \cosh \p_0, \qquad y = t \sinh \p_0, \qquad z = t^3/6 \cosh^2 \p_0 \sinh \p_0.
$$

2) Let $c \neq 0$. Denote $k = \sqrt{\frac 12 \left(1 + \frac{E}{\sqrt{1+E^2}}\right)} \in (0, 1)$, $l = \sqrt{|c|}$, $\pm 1 = \sgn c$, $\aen = \sqrt{\frac{\sqrt{1+E^2}}{2}}$. Then 
\begin{align*}
&\sinh \p = 2 \aen^2 \frac{1-k^2(1 + \sn^4 \tau)}{\cn^2 \tau}, \\
&x = 2 \aen \frac{\dn \tau \sn \tau}{l \cn \tau}, \\
&y = \pm \frac{2 \aen}{l} \left( \frac{\tau}{4 k^2 \aen^4} - \E(\tau) + \frac{\dn \tau \sn \tau}{l \cn \tau}\right), \\
&z = \pm  \frac{ \aen}{3l^3} \left(\left(\frac{\tau}{3 k^2 \aen^2} - 4 E \E(\tau)\right) \cn^3 \tau - \frac{1}{4 \aen^2 k^2} \dn \tau \sn \tau + 2 E \cn^2 \tau \dn \tau \sn \tau + 4 \aen^2 k^2 \cn^4 \tau \dn \tau \sn \tau\right), \\
&\tau = \aen l t \in [0, K(k)).
\end{align*}

\subsection{Exponential mapping}
Formulas of Subsec. \ref{subsec:norm_extr2}   parametrize the exponential mapping
\begin{align*}
&\map{\Exp}{N}{M}, \qquad \Exp(\lam, t) = \pi \circ e^{t \vH}(\lam), \\
&N = \{(\lam, t) \in C \times \R_+ \mid t \in (0, + \infty) \text{ for } c = 0; \ \ 
t \in (0, + K/(l\aen)) \text{ for } c \neq 0\}, \\
&C = T_{q_0}^*M \cap \{H = - 1/2, \ h_1 < 0 \}.
\end{align*}

\begin{proposition}
There holds the inclusion $\Exp(N) \subset \intt \B_2$. Moreover, the mapping $\map{\Exp}{N}{\intt \B_2}$ is a real-analytic diffeomorphism.
\end{proposition}
\begin{proof}
Similarly to Th. \ref{th:ExpN+-}.
\end{proof}

\subsection{Attainable set and existence of optimal trajectories}

\begin{theorem} 
We have $\A_2 = \B_2$.
\end{theorem}
\begin{proof}
Similarly to Th. \ref{th:att_set}.
\end{proof}

\begin{theorem}
Let points $q_0, q_1 \in M$ satisfy the inclusion $q_1 \in \A_2 $. Then there exists an optimal trajectory in problem \eq{pr21}--\eq{pr24}.
\end{theorem}
\begin{proof}
Similarly to Th. \ref{th:exist}.
\end{proof}

\subsection{Optimality of extremal trajectories}
Define the following function:
\begin{align*}
&\map{\tt}{C}{(0, + \infty]}, \qquad \lam = (\p_0, c) \in C, \\
&c = 0 \then \tt(\lam) = + \infty, \\
&c \neq 0  \then \tt(\lam) = \frac{K}{l \aen}.
\end{align*}

\begin{theorem}\label{th:opt4}
Let $\lam \in C$, $t_1 \in (0, \tt(\lam))$. Then the trajectory $q(t) = \Exp(\lam, t)$, $t \in [0, t_1]$, is optimal.
\end{theorem}
\begin{proof}
Similarly to Th. \ref{th:opt1}.
\end{proof}

\begin{corollary} 
For any $\lam \in C$ we have $\tcut(\lam) = \tt(\lam)$.
\end{corollary}
\begin{proof}
Let $\lam \in C$. By virtue of Th. \ref{th:opt4}, $\tcut(\lam) \geq \tt(\lam)$. On the other hand, the extremal trajectory $\Exp(\lam, t)$ is defined only for $t \in [0, \tt(\lam))$, thus $\tcut(\lam) = \tt(\lam)$.
\end{proof}

\begin{proposition}\label{prop:opt5}
Any abnormal extremal trajectory is optimal.
\end{proposition}
\begin{proof}
Similarly to Propos. \ref{prop:opt3}.
\end{proof}

\begin{theorem} 
\begin{itemize}
\item[$(1)$]
Let $q_1 \in \intt \A_2$. Then there exists a unique optimal trajectory $\Exp(\lam, t)$, $t \in [0, t_1]$, where $(\lam, t_1) = \Exp^{-1}(q_1)$. 
\item[$(2)$]
Let $q_1 = (x_1, y_1, z_1) \in \partial \A_2$, $-y_1 < x_1 \leq y_1$, $z_1 = z^2(x_1, y_1)$. Then there exists a unique optimal trajectory
corresponding to a control
\begin{align*}
&u(t) = \begin{cases}
(1, 1), \qquad &t \in [0, \tau_1], \\
(1, -1), \qquad &t \in [\tau_1, t_1], 
\end{cases}\\
&q(t) = \begin{cases}
e^{t(X_1+X_2)}(q_0), \qquad &t \in [0, \tau_1], \\
e^{(t-\tau_1)(X_1-X_2)} \circ e^{\tau_1(X_1+X_2)}(q_0), \qquad &t \in [\tau_1, t_1], 
\end{cases}\\
&\tau_1 = (y_1+x_1)/2, \quad t_1 = x_1.
\end{align*}
\item[$(3)$]
Let $q_1 = (x_1, y_1, z_1) \in \partial \A_2$, $-y_1 \leq x_1 < y_1$, $z_1 = z^1(x_1, y_1)$. Then there exists a unique optimal trajectory
corresponding to a control
\begin{align*}
&u(t) = \begin{cases}
(1, -1), \qquad &t \in [0, \tau_1], \\
(1, 1), \qquad &t \in [\tau_1, t_1], 
\end{cases}\\
&q(t) = \begin{cases}
e^{t(X_1-X_2)}(q_0), \qquad &t \in [0, \tau_1], \\
e^{(t-\tau_1)(X_1+X_2)} \circ e^{\tau_1(X_1-X_2)}(q_0), \qquad &t \in [\tau_1, t_1], 
\end{cases}\\
&\tau_1 = (y_1+x_1)/2, \quad t_1 = x_1.
\end{align*}
\end{itemize}
\end{theorem}
\begin{proof}
Similarly to Th. \ref{th:synth}.
\end{proof}

\subsection{\SL distance}

\begin{theorem}
The distance $d(q) = d(q_0, q)$ is real-analytic on $\intt \A_2$ and continuous on $M$.
\end{theorem}
\begin{proof}
Similarly to Th. \ref{th:dist}.
\end{proof}

\subsection{\SL sphere}

\begin{theorem}
The sphere $S$ is a real-analytic manifold diffeomorphic to $\R^2$ parametrized as follows: $S = \{\Exp(\lam, 1) \mid \lam \in C\}$.
\end{theorem}
\begin{proof}
Similarly to Th. \ref{th:sphere}.
\end{proof}

See the plot of two \sL spheres $S(R_1)$, $S(R_2)$ inside the attainable set $\A_2$ in Fig. \ref{fig:sph2}. 

\figout{
\onefiglabelsizen{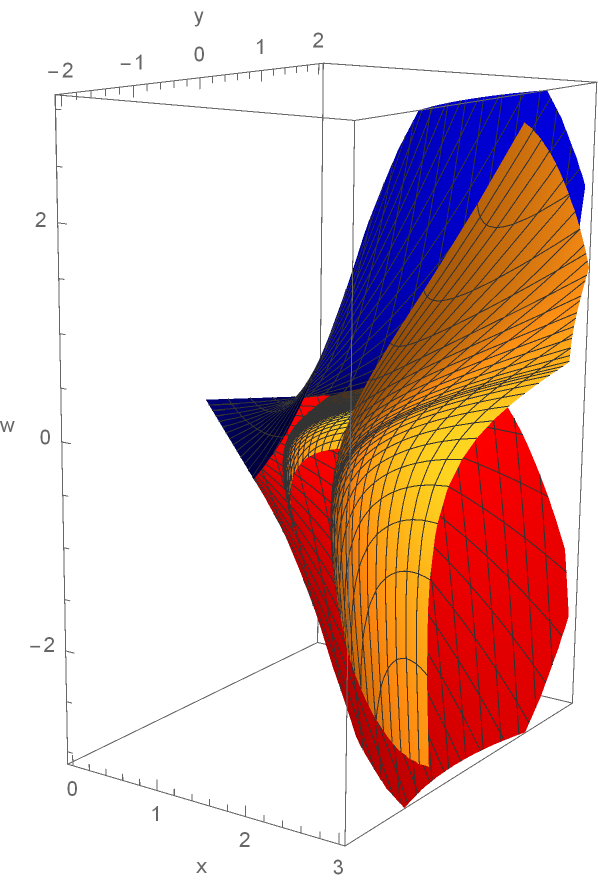}{Spheres inside the attainable set for the second problem}{fig:sph2}{7}
}

\section{Conclusion}\label{sec:conclude}  
The first problem is fundamentally different from the second one by the following properties of the
optimal synthesis:
\begin{itemize}
\item
some optimal trajectories change causal type,
\item
extremal trajectories have cut points on the Martinet surface $\Pi$,
\item
the optimal synthesis is two-valued on $\Pi$,
\item
the sub-Lorentzian distance is nonsmooth on $\Pi$ and suffers a discontinuity of the first kind at some points of the boundary of the attainable set $\partial \A_1$,
\item
the sub-Lorentzian sphere $S$ is a manifold with boundary.
\end{itemize}

These features are associated with non-trivial
intersection of the attainable set $\A_1$ (the causal future of the initial point $q_0$) and the
Martinet surface
$\Pi$ for the first problem.

The optimal synthesis in the second problem is qualitatively the same as in the \sL problem on the Heisenberg group \cite{sl_heis}.

%\newpage
\addcontentsline{toc}{section}{List of figures}
\listoffigures

%\newpage

\end{document}